\documentclass[11pt,oneside,reqno]{amsart}
\usepackage[a4paper,left=24mm,right=24mm,marginparwidth=18mm,top=28mm,bottom=34mm]{geometry}
\usepackage[all]{xy}
\usepackage[x11names]{xcolor}
\usepackage{amsmath}
\usepackage{amssymb}
\usepackage{amsthm}
\usepackage{extarrows}
\usepackage[backend=biber, style=alphabetic, maxnames=5, sorting=nyt,backref=true,doi=false,isbn=false]{biblatex}
 \AtEveryBibitem{\clearfield{url}}
 \AtEveryBibitem{\clearlist{language}}

\AtBeginBibliography{\setlength{\emergencystretch}{1em}}
\usepackage[colorlinks,linkcolor=Brown4, citecolor=Brown4]{hyperref}
\usepackage{url}
\usepackage{xurl}

\usepackage{stmaryrd,dsfont,mathrsfs}
\usepackage{MnSymbol}
\usepackage{mathtools}
\usepackage{bbm}
\usepackage[shortlabels]{enumitem}
\setlist[enumerate]{leftmargin=*}
\setlist[itemize]{labelsep=\parindent, leftmargin=2\parindent}
\usepackage{tikz-cd}
\usepackage{aliascnt}
\usepackage[nameinlink]{cleveref}
\usepackage[indent]{parskip}
\DeclarePairedDelimiterX\Set[1]\lbrace\rbrace{\def\given{\;\delimsize\vert\;}#1}

\usepackage{marginnote}

\newtheorem{thm}{Theorem}[section]
\crefname{thm}{theorem}{theorems}
\Crefname{thm}{Theorem}{Theorems}

\newtheorem{mainthm}{Theorem}

\crefname{mainthm}{theorem}{theorems}
\Crefname{mainthm}{Theorem}{Theorems}

\newaliascnt{lem}{thm}
\newtheorem{lem}[lem]{Lemma}
\aliascntresetthe{lem}
\crefname{lem}{lemma}{lemmas}
\Crefname{lem}{Lemma}{Lemmas}

\newaliascnt{cor}{thm}
\newtheorem{cor}[cor]{Corollary}
\aliascntresetthe{cor}
\crefname{cor}{corollary}{corollaries}
\Crefname{cor}{Corollary}{Corollaries}

\newaliascnt{prop}{thm}
\newtheorem{prop}[prop]{Proposition}
\aliascntresetthe{prop}
\crefname{prop}{proposition}{propositions}
\Crefname{prop}{Proposition}{Propositions}

\newaliascnt{conj}{thm}
\newtheorem{conj}[conj]{Conjecture}
\aliascntresetthe{conj}
\crefname{conj}{conjecture}{conjectures}
\Crefname{conj}{Conjecture}{Conjectures}

\newaliascnt{claim}{thm}

\aliascntresetthe{claim}
\crefname{claim}{claim}{claims}
\Crefname{claim}{Claim}{Claims}

\theoremstyle{definition}

\newaliascnt{ex}{thm}
\newtheorem{ex}[ex]{Example}
\aliascntresetthe{ex}
\crefname{ex}{example}{examples}
\Crefname{ex}{Example}{Examples}

\newaliascnt{ass}{thm}

\aliascntresetthe{ass}
\crefname{ass}{assumption}{assumptions}
\Crefname{ass}{Assumption}{Assumptions}

\newaliascnt{setup}{thm}

\aliascntresetthe{setup}
\crefname{setup}{set-up}{set-ups}
\Crefname{setup}{Set-up}{Set-ups}

\newaliascnt{notation}{thm}

\aliascntresetthe{notation}
\crefname{notation}{notation}{notations}
\Crefname{notation}{Notation}{Notations}

\newaliascnt{rem}{thm}
\newtheorem{rem}[rem]{Remark}
\aliascntresetthe{rem}
\crefname{rem}{remark}{remarks}
\Crefname{rem}{Remark}{Remarks}

\newaliascnt{defn}{thm}
\newtheorem{defn}[defn]{Definition}
\aliascntresetthe{defn}
\crefname{defn}{definition}{definitions}
\Crefname{defn}{Definition}{Definitions}

\newaliascnt{prob}{thm}

\aliascntresetthe{prob}
\crefname{prob}{problem}{problems}
\Crefname{prob}{Problem}{Problems}

\newcommand{\A}{{\mathds A}}
\newcommand{\F}{{\mathds F}}

\newcommand{\Q}{{\mathds Q}}

\newcommand{\Z}{{\mathds Z}}
\newcommand{\C}{{\mathds C}}

\newcommand\End{\mathop{\mathrm{End}}\nolimits}

\newcommand\gr{\text{\rm gr}}

\newcommand\Aut{\mathop{\mathrm{Aut}}\nolimits}

\newcommand\Cl{\mathop{\mathrm{Cl}}\nolimits}

\newcommand\Frob{\mathop{\mathrm{Frob}}\nolimits}

\newcommand\Gal{\mathop{\mathrm{Gal}}\nolimits}
\newcommand\Lie{\mathop{\mathrm{Lie}}\nolimits}
\newcommand\Hom{\mathop{\mathrm{Hom}}\nolimits}

\newcommand\GL{\mathop{\mathrm{GL}}\nolimits}

\renewcommand\ge{\mathop{\mathfrak{g}}\nolimits}

\newcommand{\fsl}{\mathop{\mathfrak{sl}}\nolimits}

\newcommand\Spin{\mathop{\mathrm{Spin}}\nolimits}
\newcommand\GSpin{\mathop{\mathrm{GSpin}}\nolimits}
\newcommand\SO{\mathop{\mathrm{SO}}\nolimits}

\newcommand\Sym{\mathop{\mathrm{Sym}}\nolimits}

\newcommand\Pic{\mathop{\mathrm{Pic}}\nolimits}

\newcommand\Rep{\mathrm{Rep}}

\newcommand\Spec{\mathop{\rm Spec}}

\newcommand\KS{\mathop{\mathrm{KS}}\nolimits}

\newcommand\K{\mathop{\mathrm{K}}\nolimits}

\newcommand{\et}{\mathrm{\acute{e}t}}

\DeclareMathOperator{\ad}{ad}

\newcommand\Vect{\mathop{\mathrm{Vect}}\nolimits}

\newcommand{\hk}{hyper-Kähler~}

\DeclareMathOperator{\Ad}{Ad}
\newcommand{\isomto}{\stackrel{\textstyle\sim}{\smash{\longrightarrow}\rule{0pt}{0.4ex}}}

\newcommand{\rR}{\mathrm{R}}

\newcommand{\rG}{\mathrm{G}}

\newcommand{\bG}{\mathbf{G}}

\newcommand{\GG}{\mathbb{G}}

\newcommand{\cO}{\mathcal{O}}

\newcommand{\scrL}{\mathscr{L}}

\DeclareMathOperator{\pst}{pst}
\DeclareMathOperator{\st}{st}

\DeclareMathOperator{\MT}{MT}
\DeclareMathOperator{\ur}{ur}

\newcommand{\OGsix}{\mathrm{OG6}}
\newcommand{\OGten}{\mathrm{OG10}}

\newcommand{\integ}{\mathrm{int}}

\newcommand{\cG}{\mathcal{G}}

\newcommand{\tisoto}{%
  \mathrel{\overset{\raisebox{-0.35ex}{$\scriptstyle\sim$}}{\rightarrow}}%
}

\newcommand{\hidden}[1]{}

\numberwithin{equation}{section}

\makeatletter
\@namedef{subjclassname@2020}{\textup{2020} Mathematics Subject Classification}
\makeatother

\title[Semisimplicity, Purity and Mumford--Tate Conjecture]{Semisimplicity, Purity and Mumford--Tate Conjecture for hyper-K\"ahler Varieties}
\date{October 6, 2026}
\author[K.~Ito]{Kazuhiro Ito}
\address{Mathematical Institute, Tohoku University, 6-3, Aoba, Aramaki, Aoba-Ku, Sendai, 980-8578, Japan}
\email{kazuhiro.ito.c3@tohoku.ac.jp}

\author[H.~Zou]{Haitao Zou}
\address{Universität Bielefeld, Universitätsstraße 25, 33615 Bielefeld, Germany}
\email{hzou@math.uni-bielefeld.de}


\begin{document}

\begin{abstract}
We prove the Mumford--Tate conjecture in every degree for hyper-K\"ahler varieties over fields finitely generated over $\Q$.
The proof establishes semisimplicity of $\ell$-adic cohomology by eliminating the unipotent radical of the algebraic monodromy group of total cohomology.
We also prove the weight--monodromy conjecture for hyper-K\"ahler varieties over $p$-adic fields.
For hyper-K\"ahler varieties over number fields with $b_2\geq4$, we establish, after suitable finite extensions, strong compatibility of the associated Weil--Deligne representations valued in Mumford--Tate groups and its integral refinement away from finitely many primes.
\end{abstract}

\maketitle

\section{Introduction}

\subsection{The Mumford--Tate conjecture and semisimplicity}
The Mumford--Tate conjecture relates the $\ell$-adic Galois
representations attached to a smooth projective variety in characteristic zero
to the Hodge structures on its Betti cohomology.
More precisely, let $X$ be a smooth projective variety over a field $F$ finitely generated over $\Q$.
Let $\overline{F}$ be an algebraic closure of $F$ and fix an embedding $\overline F\hookrightarrow\C$.
For a prime $\ell$, we write $\rG_{\ell,i}(X)$ for the algebraic monodromy group, that is, the Zariski closure of the image of the representation
\[
G_F=\Gal(\overline F/F) \longrightarrow \GL(H^i_{\et}(X_{\overline F},\Q_\ell)),
\]
and $\MT_i(X_\C)$ for the Mumford--Tate group of $H^i_B(X_\C,\Q)$.
The Mumford--Tate conjecture predicts that, under the \'etale--Betti comparison,
\[
\rG_{\ell,i}(X)^{\circ} = \MT_i(X_{\C})_{\Q_\ell}
\]
for every $i$ and $\ell$.
This conjecture was originally formulated for abelian varieties and remains open in general, even in that case; see \cite{Moonen17} and the references therein for an overview.

In this paper, we prove the Mumford--Tate conjecture in every degree for hyper-K\"ahler varieties.
Here a \emph{hyper-K\"ahler variety} over a field $K$ of characteristic zero means a smooth projective geometrically simply connected variety $X$ with $H^0(X,\Omega^2_{X/K})=K\sigma$ for a nowhere-degenerate $2$-form $\sigma$.
Hyper-K\"ahler varieties are higher-dimensional analogues of K3 surfaces and occur among the factors in the Beauville--Bogomolov decomposition \cite{Bogomolov,Beauville83}. For K3 surfaces, the Mumford--Tate conjecture was proved by Tankeev \cite{TankeevK3MTC}.
Andr\'e \cite{Andre96b} proved it in degree two for hyper-K\"ahler varieties with $b_2(X)\geq4$.
In all degrees, the conjecture was known for the four standard deformation types, $K3^{[n]}$, generalized Kummer, $\OGsix$, and $\OGten$, by work of Floccari, Soldatenkov, and Floccari--Fu--Zhang \cite{FloccariMumfordTate,Soldatenkov2022a,FFZ}. 

Tang and the second author proved the Mumford--Tate conjecture for arbitrary hyper-K\"ahler varieties after semisimplifying the Galois representations \cite{TZ26}.
The remaining problem is therefore to establish semisimplicity, or equivalently to prove that the unipotent radical of the algebraic monodromy group is trivial, which is predicted by the \emph{semisimplicity conjecture} of Grothendieck and Serre \cite{Grothendieck,SerreMotivic1994}.

We prove a general restriction on the possible unipotent radicals of algebraic monodromy groups arising from geometry: the connected algebraic monodromy group of a smooth proper variety over $F$ has no nontrivial unipotent algebraic quotient (\Cref{thm:no-unipotent quotient for geometric representation}).
For a hyper-K\"ahler variety $X$ over $F$, write $\rG_\ell(X)$ and $\MT(X_\C)$ for the algebraic monodromy and Mumford--Tate groups of total cohomology, respectively, and $\rR_{u,\ell}$ for the unipotent radical of the identity component $\rG_\ell(X)^\circ$.
The structural results of \cite{TZ26} identify $\MT(X_\C)_{\Q_\ell}$ with a Levi subgroup centralizing $\rR_{u,\ell}$, giving
\[
\rG_\ell(X)^\circ
=\MT(X_\C)_{\Q_\ell}\times\rR_{u,\ell}.
\]
The unipotent radical is therefore a quotient of $\rG_\ell(X)^\circ$ and must be trivial by the preceding result.
This proves semisimplicity of the $\ell$-adic cohomology, thereby completing the proof of the Mumford--Tate conjecture for hyper-K\"ahler varieties.
More precisely, we prove the following theorem.

\begin{mainthm}[\Cref{thm: MT conjecture}]
\label{mainthm:Mumford-Tate conjecture for HK}
\label{mainthm:semisimplicity conjecture}
Let $X$ be a hyper-K\"ahler variety over a field $F$ finitely generated over $\Q$.
Then $H^i_{\et}(X_{\overline F},\Q_\ell)$ is a semisimple $G_F$-representation for every $i$ and every prime $\ell$, and the Mumford--Tate conjecture holds for $X$ in every degree.
\end{mainthm}

The weight--monodromy conjecture in degree $i$ asserts that
the associated Weil--Deligne representation is pure of weight $i$ (see \Cref{Subsection:purity of WD representation}).
Using \Cref{mainthm:Mumford-Tate conjecture for HK}, we show that,
after a finite extension of $K$,
the tensor square
$H^i_\et(X_{\overline K},\Q_\ell)^{\otimes2}$
is a direct summand of a finite sum of tensor products
of $H^2_\et(X_{\overline K},\Q_\ell)$ and its dual.
The known degree-two case (see \Cref{lem:WMC in degree 2})
then implies that the Weil--Deligne representation associated
with this tensor square is pure of weight $2i$.
Since purity can be detected on tensor squares,
the conjecture follows in every degree.
\begin{mainthm}[\Cref{Theorem:purity of Weil-Deligne representations}]
\label{Corollary:introduction WMC for HK}
Let $K$ be a finite extension of $\Q_p$, and let $X$ be a hyper-K\"ahler variety over $K$.
The weight--monodromy conjecture holds for $H^i_\et(X_{\overline K},\Q_\ell)$ for every $i$ and every prime $\ell$, including $\ell=p$.
\end{mainthm}

\subsection{Strong compatibility}\label{intro:applications}
\label{mainthm:semisimplicity and purity}
Our next aim is to compare the local Weil--Deligne representations valued in Mumford--Tate groups as $\ell$ varies.
Let $X$ be a hyper-K\"ahler variety over a number field $F$ with $b_2(X)\geq4$.
Fix an embedding $\overline F\hookrightarrow\C$, and put $\bG=\MT(X_\C)$.
After replacing $F$ by a finite extension, the Mumford--Tate conjecture implies that the Galois actions on total cohomology factor through representations
\[
\phi^{\MT}_\ell\colon G_F\longrightarrow \bG(\Q_\ell)
\]
for every $\ell$.
For a finite place $v\mid p$ of $F$, the restriction of $\phi^{\MT}_\ell$ to $G_{F_v}$ gives rise to a $\bG$-valued Weil--Deligne representation $\rho^{\MT}_{\ell,v}$ of $F_v$ over $\overline{\Q}_\ell$ for every prime $\ell$, including $\ell=p$.
We refer to \Cref{Section: G-valued Weil-Deligne representations} for details on the notion of $\bG$-valued Weil--Deligne representations.
We fix an embedding $\overline{\Q} \hookrightarrow \overline{\Q}_\ell$ for each $\ell$.
We prove the following strong compatibility result.

\begin{mainthm}[\Cref{Theorem: strong compatibility of MT representation}]
\label{Theorem:introduction rational compatibility for HK}
After replacing $F$ by a further finite extension, the family $\{\rho^{\MT}_{\ell,v}\}_\ell$ is rationally strongly compatible for every finite place $v$ of $F$.
More precisely, there exists a $\bG$-valued Weil--Deligne representation $\rho_v$ of $F_v$ over $\overline\Q$ whose $\bG(\overline\Q)$-conjugacy class is invariant under $\Gal(\overline\Q/\Q)$ and whose base change to $\overline{\Q}_\ell$ is equivalent to $\rho^{\MT}_{\ell,v}$ for every $\ell$, including $\ell=p$.
\end{mainthm}

There has recently been substantial progress on the corresponding statement for abelian varieties.
Noot obtained partial results at places of semistable reduction \cite{Noot13}.
Kisin--Zhou proved that Frobenius conjugacy classes in the Mumford--Tate group are independent of $\ell$ at places of good reduction with odd residue characteristic \cite{KZa}.
They subsequently established strong compatibility for the Weil--Deligne representations valued in Mumford--Tate groups at places of semistable reduction, including $\ell=p$ \cite[Theorem~1.2]{KZb}.

Using a Kuga--Satake abelian variety and Looijenga--Lunts--Verbitsky (LLV) theory, we deduce our theorem from their result.
We refer to \Cref{Strongly compatible systems with integral coefficients} for the proof and a more precise description of the required finite extension of $F$.

Strong compatibility refines the classical independence of $\ell$ of Frobenius characteristic polynomials at places of good reduction \cite{WeilI,SerreMotivic1994}.
The theorem implies that, after replacing $F$ by a finite extension, the characteristic polynomials of Frobenius on the associated Weil--Deligne representations in every degree are independent of $\ell$ at every finite place, including places of bad reduction; see \Cref{Corollary: independence of Frobenius eigenvalues} for a purely local result.

We further prove an integral refinement away from finitely many primes $\ell$.
Keeping the notation above, we summarize it as follows; the precise formulation is given in \Cref{subsection:An integral refinement of strong compatibility}.
Let $\overline{\Z} \subset \overline{\Q}$ and $\overline{\Z}_\ell \subset \overline{\Q}_\ell$ be the integral closures of $\Z$ and $\Z_\ell$, respectively.

\begin{mainthm}[\Cref{Theorem:integral compatibility for HK}]
\label{Theorem:introduction integral compatibility for HK}
We keep the notation of \Cref{Theorem:introduction rational compatibility for HK} and fix a finite place $v\mid p$ of $F$.
There exist a positive integer $D$ divisible by $p$ and a reductive model $\cG$ of $\bG$ over $\Z[1/D]$ such that $\rho_v$ extends to a $\cG$-valued Weil--Deligne representation $\rho_v^{\integ}$ of $F_v$ over $\overline{\Z}[1/D]$ and, for every prime $\ell\nmid D$, its base change $(\rho_v^{\integ})_{\overline{\Z}_\ell}$ to $\overline{\Z}_\ell$ is equivalent to the $\cG$-valued Weil--Deligne representation $\rho^{\MT,\integ}_{\ell,v}$ of $F_v$ over $\overline{\Z}_\ell$ associated with the Galois action on $H^\bullet_\et(X_{\overline{F}_v},\Z_\ell)$.
\end{mainthm}

We also prove an analogous integral refinement for abelian varieties with semistable reduction at $v$.
The key input for the integral refinement is torsion-freeness of cokernels of all powers of the monodromy operator on $H^\bullet_\et(X_{\overline{F}_v},\Z_\ell)$ for all but finitely many $\ell \neq p$.
This property was studied in \cite{Ito21} in connection with a torsion analogue of the weight--monodromy conjecture.
As in the proof of the weight--monodromy conjecture, we reduce to the degree-two case to establish this property for hyper-K\"ahler varieties in every degree (\Cref{Proposition: torsion freeness of cohomological monodromy}).

\subsection*{Outline}
\Cref{Section:semisimplicity hyperkahler} shows that the connected algebraic monodromy group of a smooth proper variety over a field finitely generated over $\Q$ has no nontrivial unipotent algebraic quotient (\Cref{thm:no-unipotent quotient for geometric representation}) and deduces the Mumford--Tate conjecture and semisimplicity (\Cref{thm: MT conjecture}) for hyper-K\"ahler varieties from this result.
In \Cref{Section: G-valued Weil-Deligne representations}, we collect what we need about $\bG$-valued Weil--Deligne representations.
We also introduce the notion of purity for $\bG$-valued Weil--Deligne representations with integral coefficients, which plays an important role in the proof of \Cref{Theorem:introduction integral compatibility for HK}.
In \Cref{Section:Mumford-Tate representation and weight-monodromy}, we prove the weight--monodromy conjecture (\Cref{Theorem:purity of Weil-Deligne representations}).
Finally, in \Cref{Strongly compatible systems with integral coefficients}, we prove rational strong compatibility (\Cref{Theorem: strong compatibility of MT representation}) and its integral refinement (\Cref{Theorem:integral compatibility for HK}).

\subsection*{Notation and conventions}
For a field $F$, we write $\overline F$ for a separable closure and $G_F=\Gal(\overline F/F)$ for the absolute Galois group.
We denote base change by a subscript.
For example, for a homomorphism of rings $R\to S$, we write $X_S\coloneqq X\times_{\Spec R}\Spec S$ for an $R$-scheme $X$ and $M_S\coloneqq M\otimes_R S$ for an $R$-module $M$.
We use the same convention for group schemes, representations, etc.

\section{Mumford--Tate conjecture for hyper-K\"ahler varieties}
\label{Section:semisimplicity hyperkahler}

\subsection{Unipotent quotients of algebraic monodromy groups}
\label{Subsection:elementary semisimplicity criterion}

Let $F$ be a number field.
For a finite place $v$ of $F$, let $F_v$ be the completion of $F$ at $v$.
For $v\mid\ell$, let $\C_\ell$ be the completion of ${\overline F}_v$.
Let $\chi_\ell\colon G_{F_v}\to\Z_\ell^\times$ be the cyclotomic character, describing the action on $\ell$-power roots of unity, and write $\C_\ell(r)$ for $\C_\ell$ with action $g\cdot z=\chi_\ell(g)^r g(z)$.
A finite-dimensional continuous $G_{F_v}$-representation $W$ over $\Q_\ell$ is \emph{Hodge--Tate} if there is a $G_{F_v}$-equivariant $\C_\ell$-linear isomorphism
\[
\C_\ell\otimes_{\Q_\ell}W
\simeq\bigoplus_{j\in\Z}\C_\ell(-j)^{\oplus m_j},
\]
with the diagonal action on the left, where the $m_j$ are nonnegative integers, almost all zero.
This is an $\ell$-adic analogue of the Hodge decomposition.
The integers $j$ with $m_j>0$, counted with multiplicity $m_j$, are the \emph{Hodge--Tate weights} of $W$.
In this convention, the cyclotomic representation $\Q_\ell(1)$ has Hodge--Tate weight $-1$.
The multiplicity of each Hodge--Tate weight is additive in short exact sequences of Hodge--Tate representations.
A finite-dimensional $G_F$-representation over $\Q_\ell$ is Hodge--Tate at $v$ if its restriction to $G_{F_v}$ has this property.

Let $c \colon G_F \to \Q_{\ell}$ be a continuous additive character of $G_F$.
Denote by $V_c$ the two-dimensional representation $ g\mapsto
\begin{pmatrix}1&c(g)\\0&1\end{pmatrix}.
$ We write $I_v$ for the inertia subgroup of $G_F$ at $v$.

\begin{lem}
\label{lem:global additive rigidity}
If $V_c$ is Hodge--Tate at every place $v\mid\ell$, then $c=0$.
\end{lem}

\begin{proof}
Let $v\mid\ell$.
The representation $V_c|_{G_{F_v}}$ is an extension of the trivial representation $\Q_\ell$ by itself.
Since it is Hodge--Tate by hypothesis, additivity of the multiplicities of Hodge--Tate weights in short exact sequences shows that all its weights are zero.
Applying Sen's theorem \cite[Corollary~1]{Sen1973a} over the completion of the maximal unramified extension of $F_v$ shows that $I_v$ has finite image on $V_c$.
Thus $c(I_v)$ is finite and hence zero, since the additive group of $\Q_{\ell}$ is torsion-free.

Now let $v\nmid\ell$, and let $p$ be the residue characteristic of $F_v$.
Since $c$ is additive, its restriction to $G_{F_v}$ factors through the abelianization.
Local class field theory therefore identifies $c(I_v)$ with the image of a continuous homomorphism $\cO_{F_v}^{\times}\to \Q_{\ell}$.
The group $\cO_{F_v}^{\times}$ has an open pro-$p$ subgroup, whereas every compact subgroup of $\Q_{\ell}$ is pro-$\ell$.
Since $p\ne\ell$, this open subgroup has trivial image, so $c(I_v)$ is again finite and hence zero.

We have shown that $c$ is unramified at every finite place of $F$.
Since the target of $c$ is abelian, global class field theory makes $c$ factor through the narrow class group of $F$, which is finite.
As $\Q_{\ell}$ has no nonzero finite subgroup, we conclude that $c=0$.
\end{proof}

\begin{prop}
\label{prop:no-unipotent quotient for global representation}
Let $ \rho \colon G_F\to \GL(V) $ be a continuous finite-dimensional representation over $\Q_\ell$.
Suppose that $V$ is Hodge--Tate at every place $v \mid \ell$.
If $G$ is the algebraic monodromy group of $V$, then $G^\circ$ has no nontrivial unipotent algebraic quotient.
\end{prop}

\begin{proof}
By replacing $F$ by a finite extension, we may assume that $G=G^\circ$.
Let $c\colon G\to\GG_a$ be an algebraic homomorphism and consider the two-dimensional representation $V_c$ of $G$ as above.
By \cite[Theorem~4.14]{Milne17}, $V_c$ is a subquotient of a finite direct sum of tensor products of $V$ and its dual $V^\vee$, and hence $V_c$ is Hodge--Tate at every place $v\mid\ell$.
It then follows from \Cref{lem:global additive rigidity} that $c$ is trivial.
A nontrivial connected unipotent quotient of $G$ would admit $\GG_a$ as a quotient \cite[Propositions~14.21 and~14.22]{Milne17}, contradicting what we have just proved.
\end{proof}

\begin{thm}
\label{thm:no-unipotent quotient for geometric representation}
Let $F$ be a finitely generated field over $\Q$ and let $X$ be a smooth proper variety over $F$.
Let $\rG_\ell(X)$ be the algebraic monodromy group of the representation
\[
 \rho_\ell \colon G_F\longrightarrow \GL(H^\bullet_{\et}(X_{\overline F},\Q_\ell)).
\]
Then $\rG_\ell(X)^\circ$ has no nontrivial unipotent algebraic quotient.
\end{thm}

\begin{proof}
By the standard specialization argument using Hilbert irreducibility \cite[Section~10.6]{SerreMW1997} (see also \cite[Section~5]{CadoretMotivated2013}), it suffices to treat the case where $F$ is a number field.
In this case, $H^\bullet_{\et}(X_{\overline F},\Q_\ell)$ is Hodge--Tate at every place above $\ell$ by \cite{Faltings88,Tsuji}.
The assertion therefore follows from \Cref{prop:no-unipotent quotient for global representation}.
\end{proof}

\subsection{The Mumford--Tate conjecture and semisimplicity}
\label{Semisimple MT conjecture and its consequences}

Let $X$ be a hyper-K\"ahler variety over a field $F$ finitely generated over $\Q$, and fix an embedding $\overline F\hookrightarrow\C$.

Write $\MT(X_\C)$ for the Mumford--Tate group of total Betti cohomology and $\MT_i(X_\C)$ for that of $H^i_B(X_\C,\Q)$.
For any algebraic group acting on total Betti cohomology and preserving the grading, write $p_i$ for its degree-$i$ projection; we use the same notation on \'etale cohomology via the \'etale--Betti comparison.
In particular, $p_i(\MT(X_\C))=\MT_i(X_\C)$.

For each prime $\ell$, write $\rG_\ell(X)$ for the algebraic monodromy group of $H^\bullet_{\et}(X_{\overline F},\Q_\ell)$ and $\rR_{u,\ell}$ for the unipotent radical of $\rG_\ell(X)^\circ$.
Let $\rG_{\ell,i}(X)=p_i(\rG_\ell(X))$ be the algebraic monodromy group in degree $i$.
The Mumford--Tate conjecture $\textup{(MTC)}_i$ asserts that
\[
\rG_{\ell,i}(X)^\circ=\MT_i(X_\C)_{\Q_\ell}
\]
for every prime $\ell$.
We recall the main result of \cite{TZ26} in the following form.

\begin{thm}[\cite{TZ26}]
\label{thm: semisimple MT conjecture}
Under the comparison isomorphism
\[
H^\bullet_{\et}(X_{\overline F},\Q_\ell)
\simeq H^\bullet_B(X_\C,\Q)\otimes_\Q\Q_\ell,
\]
we have $\MT(X_\C)_{\Q_\ell}\subset\rG_\ell(X)^\circ$.
Moreover, $\rR_{u,\ell}$ commutes with $\MT(X_\C)_{\Q_\ell}$, and multiplication induces a canonical decomposition
\begin{equation}\label{eq:semisimple MTC}
\rG_\ell(X)^\circ
=
\MT(X_\C)_{\Q_\ell}\times\rR_{u,\ell}.
\end{equation}
\end{thm}

\begin{proof}
See \cite[Lemma~3.3.2 and Proposition~5.2.2]{TZ26}.
\end{proof}

The direct-product decomposition makes $\rR_{u,\ell}$ a quotient of the connected monodromy group.
The preceding subsection therefore gives the following theorem.

\begin{thm}
\label{thm: MT conjecture}
For every prime $\ell$, we have
\begin{equation}\label{eq: MT conjecture}
\rG_\ell(X)^\circ=\MT(X_\C)_{\Q_\ell}
\end{equation}
under the \'etale--Betti comparison.
Moreover, for each $0\leq i\leq 2\dim X$, $H^i_{\et}(X_{\overline F},\Q_\ell)$ is a semisimple $G_F$-representation and the Mumford--Tate conjecture $\textup{(MTC)}_i$ holds for $X$.
\end{thm}

\begin{proof}
Projection onto the second factor in \eqref{eq:semisimple MTC} makes $\rR_{u,\ell}$ a unipotent algebraic quotient of $\rG_\ell(X)^\circ$.
Thus $\rR_{u,\ell}=1$ by \Cref{thm:no-unipotent quotient for geometric representation}, which proves \eqref{eq: MT conjecture}.
For every $i$, the degree-$i$ projection satisfies $p_i(\rG_\ell(X)^\circ)=\rG_{\ell,i}(X)^\circ$.
Applying $p_i$ to \eqref{eq: MT conjecture} therefore gives
\[
\rG_{\ell,i}(X)^\circ
=p_i(\MT(X_\C)_{\Q_\ell})
=\MT_i(X_\C)_{\Q_\ell},
\]
which proves $\textup{(MTC)}_i$ in every degree.
Reductivity of Mumford--Tate groups then implies semisimplicity by \cite[Corollary~22.43]{Milne17}.
\end{proof}

\section{\texorpdfstring{$\bG$}{G}-valued Weil--Deligne representations}\label{Section: G-valued Weil-Deligne representations}

\subsection{Review of G-valued Weil--Deligne representations}\label{Subsection:G-WD representation}
Let $K$ be a complete discretely valued field of characteristic zero with residue field $\F_q$, where $q$ is a power of a prime number $p$.
Let $I_K$ be the inertia subgroup of $G_K$.
Let $W_K$ be the Weil group of $K$.
By definition, we have a short exact sequence
\[
1 \longrightarrow I_K \longrightarrow W_K \xlongrightarrow{\alpha} \Z \longrightarrow 1,
\]
where $\alpha$ sends a lift $\Frob_q$ of the geometric Frobenius element to $-1$.

Let $R$ be a $\Z[1/q]$-algebra and let $\bG$ be a reductive group scheme over $\Spec R$.
Let $\Lie(\bG)$ be the Lie algebra over $R$ of $\bG$.
Let $\Omega$ be an $R$-algebra.
In this paper, by a \emph{$\bG$-valued Weil--Deligne representation} (of $K$) over $\Omega$, we mean a pair $\rho=(r,N)$ consisting of a continuous homomorphism $ r \colon W_K \to \bG(\Omega) $ (where $\bG(\Omega)$ is endowed with the discrete topology) and an element $ N \in \Lie(\bG)_\Omega, $ which we call the monodromy operator, such that $ \Ad(r(w))(N)=q^{\alpha(w)}N $ for every $w\in W_K$.
If $\rho$ takes values in $\bG=\GL_{n}$ for some nonnegative integer $n$, we simply call $\rho$ a \emph{Weil--Deligne representation}.
Two $\bG$-valued Weil--Deligne representations $\rho_1=(r_1,N_1)$ and $\rho_2=(r_2,N_2)$ over $\Omega$ are said to be equivalent if they are conjugate by an element of $\bG(\Omega)$, that is, if there exists $g\in \bG(\Omega)$ such that $r_2(w)=\Ad(g)(r_1(w))$ for every $w\in W_K$ and $N_2=\Ad(g)(N_1)$.
In this case, we write $\rho_1 \sim_{\bG(\Omega)} \rho_2$.

For a homomorphism $\eta \colon \bG \to \bG'$ of reductive group schemes over $R$, we denote by $\eta \circ \rho$ the induced $\bG'$-valued Weil--Deligne representation $(\eta \circ r, d\eta(N))$, where $d\eta \colon \Lie(\bG) \to \Lie(\bG')$ is the homomorphism induced by $\eta$ on Lie algebras.
For a homomorphism $\iota \colon \Omega \to \Omega'$ of $R$-algebras, we denote by $\rho \otimes_\Omega \Omega'$ or $\rho_{\Omega'}$ the base change of $\rho$ along $\iota$.

If $\Omega$ is a field of characteristic zero, we say that $\rho$ is \textit{Frobenius semisimple} if $r(\Frob_q) \in \bG(\Omega)$ is semisimple for some (and hence any) geometric Frobenius lift $\Frob_q$.
For a finite extension $L/K$, the restriction \( \rho|_L \coloneqq (r|_{W_L},N) \) is Frobenius semisimple if and only if $\rho$ is Frobenius semisimple.

\subsection{G-valued Weil--Deligne representations associated with Galois representations}\label{Subsection:G-WD representation associated with Galois representation}
Let $\ell$ be a prime number and let $\bG$ be a reductive group over $\Q_\ell$.
A $\bG$-valued Weil--Deligne representation can be attached to a $\bG$-valued representation of $G_K$ (which is potentially semistable when $\ell=p$).
Here we recall the construction.

\begin{ex}\label{Example: WD rep ell adic case}
Assume that $\ell \neq p$.
Let $ \phi \colon G_K \to \bG(\Q_\ell) $ be a continuous homomorphism.
We fix an isomorphism $\Z_\ell(1) \simeq \Z_\ell$ (as $\Z_\ell$-modules).
Let $t_\ell\colon I_K \to \Z_\ell(1) \simeq \Z_\ell$ be the usual maximal pro-$\ell$ quotient.
By Grothendieck's monodromy theorem, there exists a unique nilpotent element $N \in \Lie(\bG)$ such that $\phi(\sigma)=\exp(t_\ell(\sigma)N)$ for all elements $\sigma$ of some open subgroup of $I_K$.
We fix a geometric Frobenius lift $\Frob_q \in W_K$.
We define $r \colon W_K \to \bG(\Q_\ell)$ by $\Frob^n_q \sigma \mapsto \phi(\Frob^n_q \sigma)\exp(-t_\ell(\sigma)N)$.
After passing to $\overline{\Q}_\ell$, we obtain the $\bG$-valued Weil--Deligne representation $\rho(\phi) \coloneqq (r,N)$ of $K$ over $\overline{\Q}_\ell$, whose equivalence class is independent of the choices of $\Z_\ell(1) \simeq \Z_\ell$ and $\Frob_q$ (see \cite[Lemma~8.4.3 and Variant~8.11]{Deligne73}).
\end{ex}

\begin{ex}\label{Example: WD rep p adic case}
Assume that $\ell =p$.
Let $ \phi \colon G_K \to \bG(\Q_\ell) $ be a continuous homomorphism.
We assume that $\phi$ is potentially semistable, that is, for every finite-dimensional representation $V \in \Rep_{\Q_p}(\bG)$, the induced $G_K$-action on $V$ is potentially semistable\footnote{It suffices to check this condition for one faithful representation; see \cite[Lemma 2.8]{IM20} for example.}
Let $K_0$ be the maximal unramified subfield of $K$ and $K^{\ur}_0 \subset \overline{K}$ the maximal unramified extension of $K_0$.
Every $V \in \Rep_{\Q_p}(\bG)$ gives rise to the $(\varphi,N)$-module
\[
D_{\pst}(V) \coloneqq \varinjlim_{L}(V\otimes_{\Q_p} B_{\st})^{G_L}
\]
over $K^{\ur}_0$, where $L \subset \overline{K}$ runs over all finite Galois extensions of $K$.
We choose an embedding $\tau \colon K^{\ur}_0 \hookrightarrow \overline{\Q}_p$ (over $\Q_p$).
The construction $ V \mapsto D_{\pst}(V) \otimes_{K^{\ur}_0} \overline{\Q}_p $ induces an exact tensor functor $\omega' \colon \Rep_{\Q_p}(\bG) \to \Vect_{\overline{\Q}_p}$, where $\Vect_{\overline{\Q}_p}$ is the category of finite-dimensional $\overline{\Q}_p$-vector spaces.
Let $\omega \colon \Rep_{\Q_p}(\bG) \to \Vect_{\overline{\Q}_p}$ be the canonical fiber functor $V \mapsto V_{\overline{\Q}_p}$.
By Tannakian formalism, there exists an isomorphism $\omega \simeq \omega'$ of tensor functors, and this induces an isomorphism $\bG_{\overline{\Q}_p} \simeq \underline{\Aut}^{\otimes}(\omega')$.
For each $V \in \Rep_{\Q_p}(\bG)$, the Weil group $W_K$ acts on $D_{\pst}(V)$ by
\[
  \sigma \otimes \sigma \varphi^{-\alpha(\sigma)[\F_q:\F_p]},\quad \forall\sigma \in W_K\,.
\]
This action and the monodromy operators on the $(\varphi,N)$-modules are compatible with tensor operations in $V$.
Hence, via the chosen tensor isomorphism $\omega \simeq \omega'$, they define a continuous homomorphism $ r\colon W_K\to \bG(\overline{\Q}_p) $ together with an element $ N\in \Lie(\bG)_{\overline{\Q}_p}$ such that $\rho(\phi) \coloneqq (r,N)$ is a $\bG$-valued Weil--Deligne representation of $K$ over $\overline{\Q}_p$.
The equivalence class of $\rho(\phi)$ is independent of the choices of $\tau$ and $\omega \simeq \omega'$.
\end{ex}

\begin{rem}\label{rem:functorial in groups}
   The constructions of $\bG$-valued Weil--Deligne representations $\phi \mapsto \rho(\phi)$ in 
   \Cref{Example: WD rep ell adic case,Example: WD rep p adic case} are both functorial in $\bG$, up to equivalence.
\end{rem}

For use in \Cref{subsection:An integral refinement of strong compatibility}, we record an integral version of \Cref{Example: WD rep ell adic case}.

\begin{ex}\label{Example: integral Weil-Deligne construction}
Assume that $\ell\neq p$.
Let
\[
\cG\hookrightarrow\prod_i\GL(V_i)
\]
be a reductive closed subgroup scheme over $\Z_\ell$, where $\prod_i\GL(V_i)$ is a finite product and each $V_i$ is a finite free $\Z_\ell$-module of rank $n_i$.
Assume that $\ell>\max_i n_i$.
Let $ \phi\colon G_K\to\mathcal G(\Z_\ell) $ be a continuous homomorphism, and let $\overline{\Z}_\ell$ denote the integral closure of $\Z_\ell$ in $\overline{\Q}_\ell$.
We choose an isomorphism $\Z_\ell(1)\simeq\Z_\ell$ and a geometric Frobenius lift $\Frob_q\in W_K$.
Let $(r, N)$ be the $\mathcal G$-valued Weil--Deligne representation of $K$ over $\overline{\Q}_\ell$ associated with the generic fiber of $\phi$ using our fixed choices $\Z_\ell(1)\simeq\Z_\ell$ and $\Frob_q$.
We shall show that $r \colon W_K \to \mathcal G(\overline{\Q}_\ell)$ takes values in $\mathcal G(\overline{\Z}_\ell)$ and $N$ belongs to $\Lie(\mathcal G)_{\overline{\Z}_\ell}$.

Choose $\sigma\in I_K$ with $t_\ell(\sigma)=1$, and write $g_i$ and $N_i$ for the actions of $\phi(\sigma)$ and $N$ on $V_i\otimes_{\Z_\ell}\overline{\Q}_\ell$.
The eigenvalues of $g_i$ are roots of unity, and its unipotent part is $\exp(N_i)$.
If an eigenvalue has order divisible by $\ell$, then the inequality $n_i<\ell$ forces the characteristic polynomial of $g_i$ to be irreducible over $\Q_\ell$.
Thus $g_i$ is semisimple and $N_i=0$.
If the semisimple part of $g_i$ has order $e_i$ prime to $\ell$, then $g_i^{e_i}=\exp(e_iN_i)$ is unipotent.
This, together with $(g_i^{e_i}-1)^{n_i}=0$ and $n_i<\ell$, implies that
\[
N_i=\frac{1}{e_i}\log(g_i^{e_i})
=\frac{1}{e_i}\sum_{j=1}^{n_i-1}
\frac{(-1)^{j+1}}{j}(g_i^{e_i}-1)^j
\in\End_{\overline{\Z}_\ell}(V_i\otimes_{\Z_\ell}\overline{\Z}_\ell).
\]
Similarly, $\exp(tN_i)$ is defined over $\overline{\Z}_\ell$ for every $t\in\Z_\ell$.
Since $\mathcal G$ is closed in $\prod_i\GL(V_i)$, it follows that
\[
N\in\Lie(\mathcal G)_{\overline{\Z}_\ell},
\quad
r(W_K)\subset\mathcal G(\overline{\Z}_\ell),
\]
as desired.
The resulting $\mathcal G$-valued Weil--Deligne representation of $K$ over $\overline{\Z}_\ell$ is also denoted by $\rho(\phi)=(r, N)$.
In contrast to the rational case, its equivalence class may depend on the choices of $\Z_\ell(1)\simeq\Z_\ell$ and $\Frob_q$.
The following proposition shows that it is independent of these choices under an additional hypothesis.
\end{ex}

\begin{prop}\label{Proposition: independence of integral WD construction}
Keep the notation and assumptions of \Cref{Example: integral Weil-Deligne construction}.
Assume that $\ell\nmid(q-1)$ and that there exists a geometric Frobenius lift $\Frob_q\in W_K$ such that, for every $i$, any two distinct eigenvalues $a, b$ of $\phi(\Frob_q)$ on $V_i\otimes_{\Z_\ell}\overline{\Q}_\ell$ satisfy $a-b\in \overline{\Z}_\ell^\times$.
Then the equivalence class of $\rho(\phi)$ over $\overline{\Z}_\ell$ is independent of the choices of $\Z_\ell(1)\simeq\Z_\ell$ and the geometric Frobenius lift.
\end{prop}

\begin{proof}
We follow the arguments of \cite[Lemma~8.4.3 and Variant~8.11]{Deligne73}.
If $N=0$, then $r=\phi|_{W_K}$, and the assertion is immediate.
We henceforth assume that $N\neq0$.

Since $q-1\in\Z_\ell^\times$, the same argument as in \cite[Lemma~8.4.3]{Deligne73} shows that, with the identification $\Z_\ell(1)\simeq\Z_\ell$ fixed, changing the geometric Frobenius lift yields an equivalent $\mathcal{G}$-valued Weil--Deligne representation over $\overline{\Z}_\ell$.

Changing the identification $\Z_\ell(1)\simeq\Z_\ell$ leaves $r$ unchanged and replaces $N$ by $aN$ for some $a\in\Z_\ell^\times$.
Let $x=\phi(\Frob_q)$ and let $x_s$ be its semisimple part.
By the assumption, the eigenspace decomposition of $x_s$ on each $V_i\otimes_{\Z_\ell} \overline{\Q}_\ell$ is defined over $\overline{\Z}_\ell$.
It follows that $x_s$ belongs to $\mathcal G(\overline{\Z}_\ell)$.
Since $r(I_K)$ is finite, there is an integer $m>0$ such that $x^m_s$ commutes with $r(I_K)$.
Let $D$ be the schematic closure of $\langle x_s^m\rangle$ in ${\mathcal G}_{\overline{\Z}_\ell}$.
Then $D$ is diagonalizable and commutes with $r(W_K)$.
The relation $\Ad(x_s^m)(N)=q^{-m}N$ shows that $D$ acts on $N$ via a character $\psi\colon D\to\GG_{m,\overline{\Z}_\ell}$.
Since $\psi(x_s^m)=q^{-m}$, the order of $\psi$ is infinite.
We claim that $\psi\colon D(\overline{\Z}_\ell)\to \overline{\Z}_\ell^\times$ is surjective.
Indeed, under the identification $D(\overline{\Z}_\ell)=\Hom(X^*(D),\overline{\Z}_\ell^\times)$, this map is the restriction
\[
\Hom(X^*(D),\overline{\Z}_\ell^\times)
\longrightarrow
\Hom(\Z\psi,\overline{\Z}_\ell^\times)\simeq \overline{\Z}_\ell^\times.
\]
Since $\overline{\Z}_\ell^\times$ is divisible, every homomorphism on $\Z\psi\subset X^*(D)$ extends to $X^*(D)$.
An element $h\in D(\overline{\Z}_\ell)$ with $\psi(h)=a$ therefore conjugates $(r,N)$ to $(r,aN)$.
\end{proof}

For later use, we record the following lemma.

\begin{lem}\label{Lemma: smoothness of Frobenius centralizer}
Keep the notation of
\Cref{Example: integral Weil-Deligne construction}.
Let $\rho=(r,N)$ be a $\cG$-valued Weil--Deligne representation of $K$ over $\overline{\Z}_\ell$ whose generic fiber is Frobenius semisimple.
Assume that there exists a geometric Frobenius lift $\Frob_q\in W_K$ such that, for every $i$, any two distinct eigenvalues $a, b$ of $r(\Frob_q)$ on $V_i\otimes_{\Z_\ell}\overline{\Q}_\ell$ satisfy $a-b\in\overline{\Z}_\ell^\times$.
Then the scheme-theoretic centralizer $ Z_{\cG_{\overline{\Z}_\ell}}(r(\Frob_q)) $ is smooth over $\overline{\Z}_\ell$.
Moreover, $ Z_{\cG_{\overline{\Z}_\ell}}(r(\Frob_q)) $ contains a maximal torus of $\cG_{\overline{\Z}_\ell}$.
\end{lem}

\begin{proof}
Put $x\coloneqq r(\Frob_q)$.
As in the proof of \Cref{Proposition: independence of integral WD construction}, the assumptions imply that the eigenspace decomposition of $x$ on each $V_i\otimes_{\Z_\ell} \overline{\Q}_\ell$ is defined over $\overline{\Z}_\ell$.
Let $D$ be the schematic closure of $\langle x\rangle$ in $\mathcal G_{\overline{\Z}_\ell}$.
Then $D$ is diagonalizable and
\[
Z_{\cG_{\overline{\Z}_\ell}}(x)
=Z_{\cG_{\overline{\Z}_\ell}}(D).
\]
Since $Z_{\cG_{\overline{\Z}_\ell}}(D)$ is smooth by \cite[Lemma~2.2.4]{Conrad2014}, the first assertion follows.
The second assertion follows from \cite[Remark~3.1.5 and Theorem~3.2.6]{Conrad2014}.
\end{proof}

\subsection{Purity}\label{Subsection:purity of WD representation}
Let $\rho=(r, N)$ be a Weil--Deligne representation of $K$ over $\overline{\Q}_\ell$ and let $V$ be its underlying $\overline{\Q}_\ell$-vector space.
We note that the monodromy operator $N$ is nilpotent.
Thus it induces the monodromy filtration $\{M_\bullet V\}$ on $V$; see \cite[Proposition~(1.6.1)]{WeilII}.
Let $i$ be an integer.
We say that $\rho$ is \textit{pure of weight $i$} if the eigenvalues of $r(\Frob_q)$ on $\gr_j^{M}V \coloneqq M_{j}V/M_{j-1}V$ are Weil $q^{i+j}$-numbers for any geometric Frobenius lift $\Frob_q \in W_K$.

\begin{conj}[Weight--monodromy conjecture]
    \label{Conjecture:WMC}
    Let $X$ be a smooth proper scheme over $K$.
    Then, for every prime $\ell$ and every integer $i$, the Weil--Deligne representation associated with $H^i_{\et}(X_{\overline K},\Q_\ell)$ is pure of weight $i$.
\end{conj}

The conjecture is a local analogue of the Weil conjectures:
in the good-reduction case, the monodromy operator vanishes,
and the assertion reduces to purity for the special fiber.
Several cases are known for $\ell\neq p$; see, for example,
\cite[Theorem~3.2]{Ito21} and the references therein.
(We note that the equal-characteristic analogue is already known.)
The conjecture is also known for $\ell=p$
in many of these cases; see \cite{BKV25}
and the references therein.
It remains open in general in dimension at least three.
For our argument, we need only the following low-degree case.

\begin{lem}\label{lem:WMC in degree 2}
    Let $X$ be a smooth proper scheme over $K$.
    The weight--monodromy conjecture \textup{(\Cref{Conjecture:WMC})} holds for $i\leq 2$.
\end{lem}
\begin{proof}
    For $\ell\neq p$, see \cite[Lemma~3.9]{Saito03}.
    For $\ell=p$, see \cite[Theorem~10.5]{Shi22}.
\end{proof}

Let $\bG$ be a reductive group over $\Q$.
We next extend the notion of purity to $\bG$-valued Weil--Deligne representations.
Let $\mu \colon \GG_m \to \bG$ be a cocharacter over $\Q$ whose image is contained in the center $Z(\bG)$ of $\bG$.
Let $\ell$ be a prime number, and let $\rho=(r,N)$ be a $\bG$-valued Weil--Deligne representation of $K$ over $\overline{\Q}_{\ell}$.
We say that $\rho$ is \textit{pure} (with respect to $\mu$) if for every finite-dimensional $\bG$-representation $V$ over $\overline{\Q}_{\ell}$ such that $\mu(z)$ acts on $V$ as $z^i$ for some integer $i$, the induced Weil--Deligne representation on $V$ is pure of weight $i$.

Let $\rho=(r,N)$ be a $\bG$-valued Weil--Deligne representation of $K$ over $\overline{\Q}_{\ell}$ and assume that $\rho$ is pure.
By Tannakian formalism, there exists a unique cocharacter $w \colon \GG_m \to \bG_{\overline{\Q}_{\ell}}$ with the following property: Let $V$ be a finite-dimensional $\bG$-representation over $\overline{\Q}_{\ell}$.
Since $\mu$ is central, we have a decomposition $V=\bigoplus_i V_i$ where $V_i$ is the subspace on which $\mu(z)$ acts as $z^i$.
Let $V_i=\bigoplus_j V_{i, j}$ be a decomposition into $W_K$-representations where the eigenvalues of $\Frob_q$ acting on $V_{i, j}$ are Weil $q^{i+j}$-numbers for any geometric Frobenius lift $\Frob_q \in W_K$.
Then $w(z)$ acts on $V_{i, j}$ as $z^j$.

\begin{defn}\label{Definition:weight cocharacter}
     The cocharacter $w \colon \GG_m \to \bG_{\overline{\Q}_{\ell}}$ characterized by the above property is called the \emph{Frobenius-weight cocharacter} of $r$.
\end{defn}

\begin{lem}\label{Lemma:weight cocharacter is associated with the monodromy operator}
    Let $\rho=(r,N)$ be a $\bG$-valued Weil--Deligne representation of $K$ over $\overline{\Q}_{\ell}$ and assume that $\rho$ is pure.
    Let $w \colon \GG_m \to \bG_{\overline{\Q}_{\ell}}$ be the Frobenius-weight cocharacter of $r$ and put $\lambda\coloneqq w^{-1}$.
    Then $\lambda$ is associated to $N$ in the sense of \textup{\cite[5.3, Definition]{Jantzen}}.
\end{lem}

\begin{proof}
We choose a faithful representation $V$ of $\bG_{\overline{\Q}_\ell}$.
By \cite[5.12, Claim]{Jantzen}, it suffices to show that $\lambda$ is associated to $N$ in $\GL(V)$.
We write $V=\bigoplus_{j \in \Z} V(\lambda, j)$ for the weight decomposition with respect to $\lambda$.
We have $N(V(\lambda, j))\subset V(\lambda, j+2)$ and $N^j\colon V(\lambda, -j)\tisoto V(\lambda, j)$ for $j\geq0$ by purity.
By the Jacobson--Morozov theorem, these isomorphisms give an $\fsl_2$-triple $(N,d\lambda(1),F)$ in $\End_{\overline{\Q}_\ell}(V)$.
Hence $\lambda$ is associated to $N$ in $\GL(V)$ by \cite[5.5, Proposition]{Jantzen}.
See also \cite[Proposition~(1.6.9)]{WeilII}.
\end{proof}

The following proposition generalizes \cite[Lemma~1.4(4)]{TaylorYoshida} to general $\bG$.
This also follows from \cite[Lemma~3.5]{GanHarrisSawinBeuzartPlessis}.
Here, we give a proof in terms of associated cocharacters.

\begin{prop}\label{Proposition:uniqueness of monodromy operator rational case}
Let $\rho_i=(r_i,N_i)$, for $i=1,2$, be Frobenius semisimple $\bG$-valued Weil--Deligne representations of $K$ over $\overline{\Q}_{\ell}$ which are pure with respect to $\mu$.
If $r_1$ and $r_2$ are $\bG(\overline{\Q}_{\ell})$-conjugate, then
\[
    \rho_1 \sim_{\bG(\overline{\Q}_{\ell})} \rho_2.
\]
\end{prop}

\begin{proof}
    We may assume that $r \coloneqq r_1=r_2$, and we let $w$ be the Frobenius-weight cocharacter determined by $r$.
    Let $\ge\coloneqq \Lie(\bG)_{\overline{\Q}_{\ell}}$ and let $\ge = \bigoplus_\tau \ge_\tau$ be the isotypic decomposition of $\ge$ as a $W_K$-representation, where $\tau$ runs over the isomorphism classes of irreducible $W_K$-representations over $\overline{\Q}_{\ell}$ and $\ge_\tau$ denotes the $\tau$-isotypic component.
    Put $\lambda \coloneqq w^{-1}$.
    Since $r$ commutes with $\lambda$, each component $\ge_\tau$ is stable under the action of $\lambda$.
    Hence we have the weight decomposition
\[
\ge_\tau=\bigoplus_{n\in \Z}\ge_\tau(\lambda,n),
\]
where $\lambda(z)$ acts on $\ge_\tau(\lambda,n)$ as $z^n$.

Let $P \coloneqq P_{\bG_{\overline{\Q}_{\ell}}}(\lambda)$ be the parabolic subgroup associated with $\lambda$.
We note that $\Lie(P)= \bigoplus_{\tau,n\geq0}\ge_\tau(\lambda,n)$.
By \Cref{Lemma:weight cocharacter is associated with the monodromy operator}, the cocharacter $\lambda$ is associated to both $N_1$ and $N_2$.
By \cite[5.9, Proposition]{Jantzen} (see also \cite[Lemma~5.7]{Cotner}), the map $\bigoplus_{\tau, n \geq 0}\ge_\tau(\lambda,n) \to \bigoplus_{\tau, n \geq 2}\ge_\tau(\lambda,n)$ defined by $Y \mapsto [Y, N_1]$ is surjective.
Since each component $\ge_\tau(\lambda,n)$ is mapped to $\ge_{q^{\alpha} \cdot \tau}(\lambda,n+2)$, where $q^{\alpha} \cdot \tau$ is the $W_K$-representation defined by $g \mapsto q^{\alpha(g)} \cdot \tau(g)$, we see that the induced map $\bigoplus_{n \geq 0}\ge_1(\lambda,n) \to \bigoplus_{n \geq 2}\ge_{q^{\alpha}}(\lambda,n)$ is surjective, where $1$ denotes the trivial $W_K$-representation.

Let $Z_{\bG_{\overline{\Q}_{\ell}}}(r)$ be the scheme-theoretic centralizer of $r(W_K)$.
In our setting, we have
\[
Z_{\bG_{\overline{\Q}_{\ell}}}(r)\subset Z_{\bG_{\overline{\Q}_{\ell}}}(\lambda)\subset P,
\]
and thus the equalities $Z_{\bG_{\overline{\Q}_{\ell}}}(r)\cap P=Z_{\bG_{\overline{\Q}_{\ell}}}(r)$ and $\bigoplus_{n \geq 0}\ge_1(\lambda,n) = \Lie (Z_{\bG_{\overline{\Q}_{\ell}}}(r))$ hold.
We also note that $Z_{\bG_{\overline{\Q}_{\ell}}}(r)$ is smooth since the characteristic is zero.
It follows that the orbit map $Z_{\bG_{\overline{\Q}_{\ell}}}(r) \to \bigoplus_{n \geq 2}\ge_{q^{\alpha}}(\lambda,n)$ defined by $g \mapsto \Ad(g)(N_1)$ has a dense open image.
Similarly, the orbit map $Z_{\bG_{\overline{\Q}_{\ell}}}(r) \to \bigoplus_{n \geq 2}\ge_{q^{\alpha}}(\lambda,n)$ defined by $g \mapsto \Ad(g)(N_2)$ also has a dense open image.
Since these two images are dense open subsets of the irreducible affine space $\bigoplus_{n \geq 2}\ge_{q^\alpha}(\lambda,n)$, they intersect.
In particular, $N_1$ and $N_2$ are conjugate by an element of $Z_{\bG_{\overline{\Q}_{\ell}}}(r)(\overline{\Q}_{\ell})$, as desired.
\end{proof}

\subsection{Purity with integral coefficients}
\label{Subsection:property tf}

Assume that $\bG$ has a reductive model $\cG$ over $\Z[1/D_0]$ for some positive integer $D_0$ divisible by $p$, and $\mu$ extends to a central cocharacter of this model.
Let $\ge\coloneqq \Lie(\cG)$.
Let $\ell \nmid D_0$ be a prime number.

\begin{defn}\label{Definition:integral G-valued purity}
Let $\rho=(r, N)$ be a $\cG$-valued Weil--Deligne representation of $K$ over $\overline{\Z}_\ell$.
We say that $\rho$ is \textit{pure} (with respect to $\mu$) if the following conditions hold:
\begin{enumerate}
\item The generic fiber $\rho_{\overline{\Q}_\ell}$ is pure with respect to $\mu$.
\item The Frobenius-weight cocharacter $w \colon \GG_m \to \bG_{\overline{\Q}_\ell}$ of $r$ extends over $\overline{\Z}_\ell$.
\item The scheme-theoretic centralizer $H\coloneqq  Z_{\cG_{\overline{\Z}_\ell}}(r)$ of $r(W_K)$ is smooth over $\overline{\Z}_\ell$.
\item Put $\lambda\coloneqq w^{-1}$ and
\[
\mathcal V\coloneqq 
\{
X\in\bigoplus_{n\geq2}\ge_{\overline{\Z}_\ell}(\lambda,n) \mid
\Ad(r(\sigma))(X)=q^{\alpha(\sigma)}X
\text{ for every }\sigma\in W_K
 \},
\]
where $\ge_{\overline{\Z}_\ell}=\bigoplus_{n \in \Z}\ge_{\overline{\Z}_\ell}(\lambda,n)$ is the weight decomposition with respect to $\lambda$.
Then the map $ \Lie(H)\to \mathcal V $ defined by $Y\mapsto[Y,N]$ is surjective.
\end{enumerate}
\end{defn}

This definition is motivated by the principle that purity
determines the monodromy operator, up to conjugacy, from
the underlying representation of $W_K$.
These properties hold automatically on the generic fiber
if it is pure and Frobenius semisimple, and they are used
in the proof of
\Cref{Proposition:uniqueness of monodromy operator rational case}.
The following proposition establishes the corresponding
uniqueness statement with integral coefficients.

\begin{prop}\label{Proposition:uniqueness of monodromy operator integral case}
Let $\rho_i=(r_i,N_i)$ for $i=1,2$ be $\cG$-valued Weil--Deligne representations of $K$ over $\overline{\Z}_\ell$ which are pure with respect to $\mu$ (in the sense of \Cref{Definition:integral G-valued purity}) and whose generic fibers are Frobenius semisimple.
If $r_1$ and $r_2$ are $\cG(\overline{\Z}_{\ell})$-conjugate, then
\[
    \rho_1 \sim_{\cG(\overline{\Z}_{\ell})} \rho_2.
\]
\end{prop}

\begin{proof}
Our argument is inspired by the proof of
\cite[Theorem~5.11]{Cotner}.
We may assume that $r\coloneqq r_1=r_2$.
Let $w \colon \GG_m \to \cG_{\overline{\Z}_\ell}$ be the Frobenius-weight cocharacter of $r$ and put $\lambda\coloneqq w^{-1}$.
Let $H$ and $\mathcal V$ be as in \Cref{Definition:integral G-valued purity}.
The adjoint action of $H$ preserves $\mathcal V$ and $N_1,N_2\in\mathcal V$.
We also write $\mathcal V$ for the corresponding affine space.
For $i=1,2$, consider the orbit map $ f_i\colon H \to \mathcal V, $ $h\mapsto\Ad(h)(N_i)$.
Its differential at the identity is surjective by the assumption.
Since $H$ and $\mathcal V$ are smooth, this implies that $f_i$ is smooth.
In particular, the reduction $(f_i)_{\overline{\F}_\ell}$ has a dense open image for $i=1, 2$.
These two images intersect, and there is an element $\overline h\in H(\overline{\F}_\ell)$ such that $\Ad(\overline h)(N_1 \otimes_{\overline{\Z}_\ell} \overline{\F}_\ell)=N_2 \otimes_{\overline{\Z}_\ell} \overline{\F}_\ell$.
The fiber $f_1^{-1}(N_2)$ is smooth over $\overline{\Z}_\ell$ and has the $\overline{\F}_\ell$-point $\overline h$.
Since $\overline{\Z}_\ell$ is henselian, this point lifts to an element $h\in H(\overline{\Z}_\ell)$ satisfying $\Ad(h)(N_1)=N_2$.
As $h$ centralizes $r$, it conjugates $\rho_1$ to $\rho_2$.
\end{proof}

The following lemma gives a sufficient criterion for integral
purity in terms of torsion-freeness of the cokernel of
the adjoint action of $N$, which will be used in
\Cref{subsection:An integral refinement of strong compatibility}.

\begin{lem}\label{Lemma: sufficient criterion for integral purity}
Let $\rho=(r,N)$ be a $\cG$-valued Weil--Deligne representation of $K$ over $\overline{\Z}_\ell$ whose generic fiber is Frobenius semisimple.
Assume that $\rho$ satisfies conditions (1)--(3) of \Cref{Definition:integral G-valued purity}.
Let $ \ge_{\overline{\Q}_\ell} =\bigoplus_\tau\ge_{\overline{\Q}_\ell,\tau} $ be the isotypic decomposition for the $W_K$-action, and put $ \ge_{\overline{\Z}_\ell,\tau} \coloneqq \ge_{\overline{\Z}_\ell} \cap\ge_{\overline{\Q}_\ell,\tau}. $ Assume that
\[
\ge_{\overline{\Z}_\ell}
=\bigoplus_\tau\ge_{\overline{\Z}_\ell,\tau}
\]
and that the cokernel of
\[
d \colon
\ge_{\overline{\Z}_\ell}\longrightarrow
\ge_{\overline{\Z}_\ell}, \quad Y \mapsto [Y,N]
\]
is torsion-free.
Then $\rho$ is pure with respect to $\mu$.
\end{lem}

\begin{proof}
It remains to verify condition (4).
Let $\lambda\coloneqq w^{-1}$ and let $\ge_{\overline{\Z}_\ell,\tau}= \bigoplus_{n \in \Z} \ge_{\overline{\Z}_\ell,\tau}(\lambda, n)$ be the weight decomposition with respect to $\lambda$.
As in the proof of \Cref{Proposition:uniqueness of monodromy operator rational case},
\[
d \colon \bigoplus_{n\geq0}
\ge_{\overline{\Z}_\ell,1}(\lambda,n)
\longrightarrow
\bigoplus_{n\geq2}
\ge_{\overline{\Z}_\ell,q^\alpha}(\lambda,n)
\]
is surjective after tensoring with $\overline{\Q}_\ell$.
Its cokernel is a torsion-free $\overline{\Z}_\ell$-module by the assumptions, and hence it is zero.
As $\lambda$ is central in $H$, the source is $\Lie(H)$, and the target is $\mathcal V$.
This shows that $\rho$ satisfies condition (4).
\end{proof}

\section{Weight--monodromy conjecture for hyper-K\"ahler varieties}
\label{Section:Mumford-Tate representation and weight-monodromy}

Throughout this section, let $K$ be a finite extension of $\Q_p$ and let $X$ be a \hk variety over $K$.
We fix an embedding $\overline{K} \hookrightarrow \C$.
Let $\bG\coloneqq\MT(X_\C)$ and $\bG_2 \coloneqq \MT_2(X_\C)$.

In this section, we prove the weight--monodromy conjecture for $X$ as an application of \Cref{thm: MT conjecture}.
After replacing $K$ by a finite extension, the Mumford--Tate conjecture enables us to construct the associated $\bG$-valued Weil--Deligne representations.
Using these representations and the properties of the projection $\bG\to\bG_2$, we deduce the weight--monodromy conjecture in every degree from the degree-two case, which is already known.
We also prove Frobenius semisimplicity by a similar reduction to the degree-two case.

\subsection{Weil--Deligne representations valued in Mumford--Tate groups}
\label{Subsection:MT-WD representation}

We choose a model of $X$ over a finitely generated subfield $F\subset K$.
Let $\overline F\subset\overline K$ be the algebraic closure of $F$.
Let $F^c/F$ be a finite extension in $\overline F$ such that for every $\ell$, the algebraic monodromy group associated with the representation $G_{F^c} \to \GL(H^\bullet_\et(X_{\overline{K}}, \Q_\ell))$ is connected.
Such an extension exists by a theorem of Serre \cite{RibetLetter}; see also \cite[Proposition~(6.14)]{LarsenPink92}.
Let $K^c\coloneqq KF^c$.
Then \Cref{thm: MT conjecture} implies that the action of $G_{K^c}$ on $H^\bullet_\et(X_{\overline{K}}, \Q_\ell)$ factors through $\bG(\Q_\ell)$ for every $\ell$ (including $\ell=p$) under the \'etale--Betti comparison.
We denote this representation by
\[
\phi^{\MT}_\ell \colon G_{K^c}\longrightarrow \bG(\Q_\ell)
\]
and the associated $\bG$-valued Weil--Deligne representation of $K^c$ over $\overline{\Q}_\ell$ by $\rho^{\MT}_\ell$.
Let $\rho_{\ell, i}$ be the Weil--Deligne representation of $K$ over $\overline{\Q}_\ell$ associated with $H^i_\et(X_{\overline{K}}, \Q_\ell)$.
For each $i$, let $p_i\colon\bG\to\GL(H^i_B(X_\C,\Q))$ be the natural representation.
Then $p_i\circ\rho^{\MT}_\ell$ is equivalent to the restriction $\rho_{\ell, i}|_{K^c}$.

Let $s\in\bG(\Q)$ be the \emph{parity operator}, obtained by evaluating the weight cocharacter at $-1$; it acts on $H^i_B(X_\C,\Q)$ as $(-1)^i$.
We use the following consequence of LLV theory: the degree-two projection
\[
 p_2\colon\bG\longrightarrow\bG_2
\]
is a central isogeny with kernel $\langle s \rangle$; see \cite[Lemma~3.2.5]{TZ26}.
This kernel is trivial when the odd cohomology vanishes and has order two otherwise.

Using these facts, we can prove the following result.

\begin{prop}\label{prop:global Frobenius semisimplicity}
For every prime $\ell$ and every $i$, the Weil--Deligne representation $\rho_{\ell, i}$ of $K$ over $\overline{\Q}_\ell$ associated with $H^i_\et(X_{\overline{K}}, \Q_\ell)$ is Frobenius semisimple.
The $\bG$-valued Weil--Deligne representation $\rho^{\MT}_\ell$ is also Frobenius semisimple.
\end{prop}

\begin{proof}
Since Frobenius semisimplicity is invariant under finite extensions of $K$, we may assume that $K=K^c$.
We first show that $\rho_{\ell, 2}$ is Frobenius semisimple for every $\ell$.

If $b_2(X)\geq 4$, the Kuga--Satake construction \cite{Andre96b} (see also \cite[Lemma~4.6]{ArithmeticNagai} and \Cref{subsection:Strong compatibility of Mumford--Tate-valued Weil representations}) reduces the assertion to the corresponding result for abelian varieties, which follows from \cite[Expos\'e~IX]{SGA7-I} when $\ell\neq p$ and from \cite{Coleman-Iovita} when $\ell=p$.

If $b_2(X)=3$, then $H^{1,1}(X_\C)$ is spanned by a polarization whose orthogonal complement in $H^2_B(X_\C,\Q)$ has rank two.
The corresponding special orthogonal group is a torus, which in turn implies that $\bG_2$ is a torus.
This implies the assertion.

Now write $\rho^{\MT}_\ell=(r^{\MT}_\ell,N^{\MT}_\ell)$, and let $x$ be the value of $r^{\MT}_\ell$ at a geometric Frobenius lift.
Write $x=x_sx_u$ for its Jordan decomposition.
Since $p_2(x)$ is semisimple, we have $p_2(x_u)=1$.
The kernel of $p_2$ contains no nontrivial unipotent elements, so $x_u=1$.
Therefore $\rho^{\MT}_\ell$ is Frobenius semisimple.
This implies that $\rho_{\ell, i}$ is also Frobenius semisimple for every $i$.
\end{proof}

\subsection{Weight--monodromy conjecture for hyper-K\"ahler varieties}\label{Purity and the weight--monodromy conjecture}

We now prove the weight--monodromy conjecture for \hk varieties.
The degree-two case is already known by \Cref{lem:WMC in degree 2}.
To pass to higher degrees, we use the following construction.

\begin{lem}\label{lem:weighted tensor generation}
    Let $V_i=H^i_B(X_\C,\Q)$ for each integer $0\leq i\leq 2\dim X$.
    \begin{enumerate}
        \item The action of $\bG$ on $V^{\otimes 2}_i\coloneqq V_i\otimes V_i$ factors through $\bG_2$.
        \item As a $\bG_2$-representation, $V^{\otimes 2}_i$ is a direct summand of a finite direct sum of representations $ V_2^{\otimes m}\otimes  V_2^{\vee,\otimes n} $ for some $m,n\geq 0$ with $m-n=i$.
    \end{enumerate}
\end{lem}
\begin{proof}
(1) The kernel of $p_2$ is generated by $s$, which acts on $V_i$ as $(-1)^i$ and hence trivially on $V^{\otimes 2}_i$.
The assertion follows since $p_2$ is surjective.

(2) The group $\bG_2$ is reductive, and $V_2$ is faithful as a $\bG_2$-representation.
Thus $V^{\otimes 2}_i$ is a direct summand of a finite direct sum of $ V_2^{\otimes m}\otimes  V_2^{\vee,\otimes n} $ for some $m,n\geq 0$ by \cite[Theorems~4.14 and~22.42]{Milne17}.
The scalar subgroup $\GG_m\subset\bG_2$ acts on
$V_i^{\otimes2}$ as $z^i$ and on
$V_2^{\otimes m}\otimes V_2^{\vee,\otimes n}$
as $z^{m-n}$.
Thus we may assume that $m-n=i$, as required.
\end{proof}

The following lemma detects purity from a tensor square.

\begin{lem}\label{lem:tensor-square defector of purity}
    Let $\rho=(r,N)$ be a Weil--Deligne representation of $K$ over $\overline{\Q}_\ell$.
    If $\rho\otimes\rho$ is pure of weight $2i$ for some integer $i$, then $\rho$ is pure of weight $i$.
\end{lem}
\begin{proof}
Let $V$ be the underlying vector space of $\rho$ and let $M_\bullet V$ be the monodromy filtration of $N$.
The monodromy operator on $V\otimes V$ is $N\otimes 1+1\otimes N$.
Its monodromy filtration is the convolution of the two filtrations, so
\[
 \gr_m^M(V\otimes V)\simeq\bigoplus_{a+b=m}\gr_a^M V\otimes\gr_b^M V.
\]
This follows, for instance, from the Jacobson--Morozov theorem; see \cite[Proposition~(1.6.9)]{WeilII}.

Let $a$ be an eigenvalue of geometric Frobenius on $\gr_j^M V$.
Then $a^2$ is an eigenvalue on the summand $\gr_j^M V\otimes\gr_j^M V$ of $\gr_{2j}^M(V\otimes V)$.
Purity of $V\otimes V$ says that $a^2$ is a Weil $q^{2(i+j)}$-number.
It follows that $a$ is a Weil $q^{i+j}$-number, as desired.
\end{proof}

We can now prove the weight--monodromy conjecture (\Cref{Conjecture:WMC}) for hyper-K\"ahler varieties.

\begin{thm}
\label{Theorem:purity of Weil-Deligne representations}
Let $X$ be a hyper-K\"ahler variety over $K$.
For every prime $\ell$ and every integer $0\leq i\leq 2\dim X$, the Weil--Deligne representation $\rho_{\ell, i}$ is pure of weight $i$.
\end{thm}
\begin{proof}
Since purity is invariant under finite extensions, we may assume that $K=K^c$.
Using \Cref{lem:weighted tensor generation} and the \'etale--Betti comparison, we see that $H^i_\et(X_{\overline{K}}, \Q_\ell)^{\otimes 2}$ is a direct summand of a finite direct sum of $ H^2_\et(X_{\overline{K}}, \Q_\ell)^{\otimes m}\otimes H^2_\et(X_{\overline{K}}, \Q_\ell)^{\vee,\otimes n} $ with $m-n=i$ as a $\bG_2$-representation over $\Q_\ell$.
Since the action of $G_K$ on total cohomology factors through $\phi^{\MT}_\ell\colon G_K\to\bG(\Q_\ell)$, it follows that this splitting is also $G_K$-equivariant.
Thus $\rho_{\ell,i}^{\otimes 2}$ is a direct summand of a finite direct sum of Weil--Deligne representations $ \rho_{\ell,2}^{\otimes m}\otimes \rho_{\ell,2}^{\vee,\otimes n} $ with $m-n=i$.
By \Cref{lem:WMC in degree 2}, $\rho_{\ell,2}$ is pure of weight $2$, so each $ \rho_{\ell,2}^{\otimes m}\otimes \rho_{\ell,2}^{\vee,\otimes n} $ is pure of weight $2(m-n)=2i$.
Hence $\rho_{\ell,i}^{\otimes 2}$ is pure of weight $2i$.
The assertion now follows from \Cref{lem:tensor-square defector of purity}.
\end{proof}

\section{Strongly compatible systems with rational and integral coefficients}
\label{Strongly compatible systems with integral coefficients}

In this section, we establish strong compatibility for the Weil--Deligne representations valued in Mumford--Tate groups attached to hyper-K\"ahler varieties over number fields and its integral refinement.

\subsection{Compatibility of Weil--Deligne representations}
\label{Subsection:definitions strong compatibility}

Let $F$ be a number field, and fix an embedding $\overline F\hookrightarrow\C$.
Let $X$ be an abelian variety or a \hk variety over $F$.
Put
\[
V_B=
\begin{cases}
H^1_B(X_\C,\Q), & \text{if $X$ is an abelian variety},\\
H^\bullet_B(X_\C,\Q), & \text{if $X$ is a \hk variety},
\end{cases}
\]
and, for every prime $\ell$, put
\[
V_\ell=
\begin{cases}
H^1_{\et}(X_{\overline F},\Q_\ell),
  & \text{if $X$ is an abelian variety},\\
H^\bullet_{\et}(X_{\overline F},\Q_\ell),
  & \text{if $X$ is a \hk variety}.
\end{cases}
\]
Let $\bG=\MT(X_\C)$ denote the Mumford--Tate group of $V_B$.
Write
\[
\mu\colon\GG_m\longrightarrow Z(\bG)
\]
for the weight cocharacter, normalized so that $\mu(z)$ acts as $z$ on $V_B$ if $X$ is an abelian variety and as $z^i$ on $H^i_B(X_\C,\Q)$ for each $i$ if $X$ is a \hk variety.
By replacing $F$ by a finite extension, we assume that the algebraic monodromy group of $V_\ell$ is connected for every prime $\ell$.
By Deligne's theorem on absolute Hodge cycles for abelian varieties \cite{Deligne} and by \Cref{thm: MT conjecture} for \hk varieties, the Galois representation on $V_\ell$ factors through
\[
\phi_\ell^{\MT} \colon G_F\longrightarrow \bG(\Q_\ell),
\]
where we identify $V_\ell$ with $V_B\otimes_\Q\Q_\ell$ via the \'etale--Betti comparison.

Let $v \mid p$ be a finite place of $F$.
Applying the construction of \Cref{Subsection:G-WD representation associated with Galois representation} to $\phi_\ell^{\MT}|_{G_{F_v}}$ gives a $\bG$-valued Weil--Deligne representation
\[
\rho_{\ell,v}^{\MT}
=(r_{\ell,v}^{\MT},N_{\ell,v})
\]
of $F_v$ over an algebraic closure $\overline{\Q}_\ell$ of $\Q_\ell$ for every prime $\ell$, including $\ell=p$.

Let $\overline{\Q}\subset\C$ be the algebraic closure of $\Q$ in $\C$, and fix an embedding $\iota_\ell \colon \overline{\Q}\hookrightarrow\overline{\Q}_\ell$ for every $\ell$.

\begin{defn}
\label{Definition:rational strong compatibility}
The family $\{\rho^{\MT}_{\ell,v}\}_\ell$ is \emph{rationally strongly compatible} if there exists a $\bG$-valued Weil--Deligne representation $\rho_v$ of $F_v$ over $\overline{\Q}$ satisfying the following conditions:
\begin{enumerate}
    \item ${}^{\sigma}\rho_v \sim_{\bG(\overline{\Q})}\rho_v$ for every $\sigma \in G_\Q$, where ${}^{\sigma}\rho_v$ denotes the base change of $\rho_v$ along $\sigma \colon \overline{\Q} \tisoto \overline{\Q}$.
    \item $\rho_v \otimes_{\overline{\Q}} \overline{\Q}_\ell \sim_{\bG(\overline{\Q}_\ell)} \rho^{\MT}_{\ell,v}$ for every $\ell$, including $\ell=p$.
\end{enumerate}
\end{defn}

\begin{rem}
\label{Remark:difference Kisin-Zhou}
Extension of scalars from $\overline{\Q}$ to $\C$ induces a bijection between the $G_\Q$-invariant equivalence classes of $\bG$-valued Weil--Deligne representations of $F_v$ over $\overline{\Q}$ and the $\Aut(\C/\Q)$-invariant equivalence classes of $\bG$-valued Weil--Deligne representations of $F_v$ over $\C$.
The latter are called ``defined over $\Q$'' in \cite[Definition~4.1.7]{KZb}.
Indeed, after fixing a finite quotient through which the inertia subgroup acts, $\bG$-valued Weil--Deligne representations are parametrized by an affine scheme of finite type over $\Q$.
An $\Aut(\C/\Q)$-invariant $\bG$-conjugacy orbit is locally closed and descends to $\overline{\Q}$, so it has a $\overline{\Q}$-point.
Moreover, two $\bG$-valued Weil--Deligne representations over $\overline{\Q}$ that are equivalent over $\C$ are already equivalent over $\overline{\Q}$ since their transporter is a scheme of finite type over $\overline{\Q}$, and hence has a $\overline{\Q}$-point whenever it has a $\C$-point.
This proves the desired bijection.

Thus \Cref{Definition:rational strong compatibility} is equivalent to the formulation of strong compatibility used in \cite{KZb}.
\end{rem}

\begin{thm}[{\cite[Theorem~1.2]{KZb}}]
\label{Theorem:strongly compatiblity for abelian variety}
In the above setting, suppose that $X$ is an abelian variety with semistable reduction at $v$.
Then $\{\rho^{\MT}_{\ell,v}\}_\ell$ is rationally strongly compatible.
\end{thm}

\begin{proof}
This is \cite[Theorem~1.2]{KZb}, reformulated using \Cref{Remark:difference Kisin-Zhou}.
\end{proof}

\subsection{Strong compatibility for \hk varieties}\label{subsection:Strong compatibility of Mumford--Tate-valued Weil representations}

Now we assume that $X$ is a \hk variety over $F$ with $b_2(X) \geq 4$.
We shall deduce from \Cref{Theorem:strongly compatiblity for abelian variety} its analogue for $X$ after replacing $F$ by a finite extension.
For this, we use the Kuga--Satake construction \cite{Deligne72, Andre96b}.

The $\Q$-vector space $H^2_B(X_\C, \Q)(1)$ carries a natural non-degenerate $\Q$-quadratic form $q$, which is called the Beauville--Bogomolov--Fujiki (BBF) form.
We denote this quadratic space over $\Q$ by $V^2_{B}(1)$ and let $\Cl(V^2_{B}(1))$ be the associated Clifford algebra over $\Q$.
Let $\GSpin\coloneqq \GSpin(V^2_B(1))$ be the associated general spin group and let $\SO\coloneqq\SO(V^2_B(1))$.
We write
\[
\pi \colon \GSpin\longrightarrow \SO,
\quad
\nu \colon \GSpin\longrightarrow\GG_m
\]
for the natural projection and the spinor norm, respectively.

In this paper, we employ the following definition.

\begin{defn}\label{Definition: Kuga-Satake abelian variety}
Let $E \subset \overline{F}$ be a finite extension of $F$.
We say that $X$ \textit{admits a Kuga--Satake abelian variety over $E$ of dimension $2^{b_2(X)-1}$} if there exists an abelian variety $A$ over $E$ with an isomorphism of $\Q$-vector spaces
\[
\theta \colon H^1_B(A_\C, \Q) \isomto\Cl(V^2_{B}(1))
\]
satisfying the following conditions:
\begin{enumerate}
\item Let
\[
\GSpin
\hookrightarrow \GL(H^1_B(A_\C,\Q))
\]
be the injection obtained via $\theta$ from left multiplication on $\Cl(V^2_{B}(1))$.
The Hodge cocharacter $\mu_{A} \colon \GG_{m, \C} \to \GL(H^1_B(A_{\C}, \C))$ factors through $\GSpin_\C$, and the composition $\GG_{m, \C} \xrightarrow{\mu_{A}} \GSpin_\C \xrightarrow{\pi} \SO_\C$ is equal to the Hodge cocharacter associated with $V^2_B(1)$.

\item For every prime $\ell$, the Galois representation
\[
\phi^{\KS}_{\ell} \colon G_E \longrightarrow \GL(H^1_\et(A_{\overline{F}}, \Q_\ell))
\]
factors through $\GSpin(\Q_\ell)$ (via the \'etale--Betti comparison), and via $\pi \colon \GSpin_{\Q_\ell} \to \SO_{\Q_\ell}$, this induces the usual $G_E$-action on $V^2_\ell(1)\coloneqq H^2_\et(X_{\overline{F}}, \Q_\ell)(1)$.
Moreover, the composition $G_E 
\xrightarrow{\phi^{\KS}_{\ell}} \GSpin(\Q_\ell) \xrightarrow{\nu} \Q^\times_\ell$ is the inverse of the cyclotomic character $\chi_\ell$.
\end{enumerate}
We call such an abelian variety $A$, equipped with $\theta$, a \emph{Kuga--Satake abelian variety of $X$ over $E$}.
\end{defn}

As recalled in \cite[Section~4.2]{ArithmeticNagai}, the Kuga--Satake construction \cite{Deligne72,Andre96b} provides a Kuga--Satake abelian variety of $X$ over some finite extension $E/F$ in the sense above.
In fact, a result of Andr\'e \cite[Theorem~8.4.3]{Andre96b} allows us to control the extension $E/F$ as follows.

\begin{lem}\label{Lemma: existence of KS abelian variety}
Let $\scrL$ be an ample line bundle on $X$ and put
\[
\Lambda_{\widehat\Z}
\coloneqq H^2_\et(X_{\overline F},\widehat\Z(1)),
\quad
P_{\widehat\Z}
\coloneqq c_1(\scrL)^\perp\subset\Lambda_{\widehat\Z},
\]
where the orthogonal complement is taken with respect to the BBF form.
Write $P_{\A_f} \coloneqq P_{\widehat\Z}\otimes_{\widehat\Z}\A_f$ for finite adèles $\A_f$ of $\Q$.
For $m=3$ or $m=4$, let $\K_m^{\ad}$ be the image of
\[
\K_m\coloneqq 
\Set*{
g\in
\GSpin(P_{\A_f})\cap\Cl(P_{\widehat\Z})^\times
\given g-1\in m\Cl(P_{\widehat\Z})
}
\]
under the natural projection to $\SO(P_{\widehat\Z})$.
Let $E/F$ be a finite extension such that the image of $G_E \to \mathrm{O}(P_{\widehat\Z})$ is contained in $\K_m^{\ad}$.
Then $X$ admits a Kuga--Satake abelian variety $A$ over $E$ in the sense of \Cref{Definition: Kuga-Satake abelian variety}.
Moreover, $A$ can be chosen to have semistable reduction at every finite place of $E$ of residue characteristic prime to $m$.
\end{lem}

\begin{proof}
By \cite[Theorem~8.4.3]{Andre96b}, the assumption ensures that the Kuga--Satake construction associated with the even Clifford algebra of the primitive cohomology can be carried out over $E$, yielding a Kuga--Satake package in the sense of \cite[Definition~4.5.1]{Andre96b} such that all $m$-torsion points of the underlying abelian variety $B/E$ are rational over $E$.
(Strictly speaking, we use a different quadratic form from the one in \cite{Andre96b}, but the same argument applies.)
We set $A\coloneqq B^4$.
Using \cite[Variant~4.1.3 and Corollary~6.4.4]{Andre96b}, together with \cite[Lemma~4.6]{ArithmeticNagai}, one can check that there is an isomorphism $\theta \colon H^1_B(A_\C, \Q) \tisoto \Cl(V^2_{B}(1))$ such that $A$, equipped with $\theta$, is a Kuga--Satake abelian variety over $E$ in the sense of \Cref{Definition: Kuga-Satake abelian variety}.

Since all $m$-torsion points of $A$ are $E$-rational, it follows from Raynaud's criterion \cite[Expos\'e~IX, Proposition~4.7]{SGA7-I} that $A$ has semistable reduction at every finite place of $E$ of residue characteristic prime to $m$.
\end{proof}

Assume now that $X$ admits a Kuga--Satake abelian variety $A$ over $E$ of dimension $2^{b_2(X)-1}$.
To relate the higher cohomology of $X$ to the cohomology of $A$, we use the twisted Looijenga--Lunts--Verbitsky (LLV) representation
\[
R \colon \GSpin
\longrightarrow \prod_i \GL(H^i_B(X_\C,\Q)).
\]
Here the restriction of $R$ to $\Spin(V^2_B(1))$ integrates the action of the degree-zero semisimple part of the LLV Lie algebra, and $z\in\GG_m\subset\GSpin$ acts in degree $i$ as $z^i$.
We refer to \cite[Section~3.2]{ArithmeticNagai} and the references therein for the construction.
It follows from a theorem of Verbitsky \cite[Theorem~1.4]{Verbitskyaction} that $\bG=\MT(X_\C)$ is contained in the image $R(\GSpin)$; see also \cite[Theorem~3.21 and Remark~3.22]{ArithmeticNagai}.

\begin{lem}\label{Lemma: LLV and Kuga-Satake hodge structure}
The composition $\GG_{m,\C}\xrightarrow{\mu_A}\GSpin_\C\xrightarrow{R}\prod_i\GL(H^i_B(X_\C,\C))$ is the Hodge cocharacter $\mu_X$ of $H^\bullet_B(X_\C,\Q)$.
In particular, we have the following commutative diagram:
\[
\begin{tikzcd}
\MT(A_{\C}) \arrow[r] \arrow[d, hook]
&
\bG \arrow[d, hook]
\\
\GSpin \arrow[r, "R"]
&
R(\GSpin) \arrow[r, hook]
&
\displaystyle \prod_i \GL(H^i_B(X_{\C},\Q)).
\end{tikzcd}
\]
The upper horizontal map is the restriction of $R$ and is an isogeny.
\end{lem}

We denote the restriction $\MT(A_\C) \to \bG$ by the same notation $R$.

\begin{proof}
Both $\mu_X$ and $R\circ\mu_A$ take values in $R(\GSpin)_\C$.
Since the projection
\[
p_2 \colon \prod_i \GL(H^i_B(X_{\C}, \Q)) \longrightarrow \GL(H^2_B(X_{\C}, \Q))
\]
    induces a central isogeny from $R(\GSpin)$ onto its image, it is enough to show that their degree-two projections are the same.
    Since $p_2\circ R=\nu\cdot\pi$, the required equality follows immediately.
    By the defining property of Mumford--Tate groups, the equality $R\circ\mu_A=\mu_X$ implies $R(\MT(A_\C))=\bG$.
    Since $\nu\cdot\pi$ has finite kernel, so does $R$.
    Thus the restriction $R\colon\MT(A_\C)\to\bG$ is an isogeny.
\end{proof}

Recall that we assume that the algebraic monodromy group of the action of $G_F$ on $H^\bullet_\et(X_{\overline F},\Q_\ell)$ is connected for every prime $\ell$.

\begin{prop}\label{Proposition:LLV and Kuga-Satake}
Keep the notation and assumptions above.
Let $s \in Z(\bG)(\Q)$ be the element which acts as $(-1)^i$ on $H^i_B(X_\C,\Q)$.
\begin{enumerate}
    \item For every $\ell$, the homomorphism $\phi^{\KS}_\ell \colon G_E \to \GSpin(\Q_\ell)$ factors through $\MT(A_\C)(\Q_\ell)$.
    \item There is a unique continuous character $\epsilon \colon G_E \to \langle s \rangle =\{ 1, s \}$ such that the following diagram commutes for every $\ell$:
\[
\begin{tikzcd}
G_E
\arrow[r, "\phi^{\KS}_{\ell}"]
\arrow[dr, "\epsilon\phi^{\MT}_{\ell}"']
&
\MT(A_{\C})(\Q_\ell)
\arrow[d, "R"]
\\
&
\bG(\Q_\ell).
\end{tikzcd}
\]
\end{enumerate}
\end{prop}

\begin{proof}
    (1) We choose an ample line bundle on $X$, and write $\eta$ for its class in $V^2_B(1)$.
    Let $H \subset \GSpin$ be the stabilizer of $\eta$ with respect to its usual action on $V^2_B(1)$ via $\pi$.
    By the properties of the Kuga--Satake abelian variety, we have $\MT(A_\C)\subset H$ and $\phi^{\KS}_\ell(G_E)\subset H(\Q_\ell).$ We define $f \colon H\to\GL(H^2_B(X_\C,\Q))$ by $f=(p_2\circ R)|_H =(\nu\cdot\pi)|_H$.
    Then the kernel of $f$ is $\{ \pm 1 \} \subset H$.
    Let $\MT_2(X_\C) \subset \GL(H^2_B(X_\C, \Q))$ be the Mumford--Tate group of degree $2$.
    Since $\MT(A_{\C})$ contains $\{ \pm 1 \}$ and $f(\MT(A_{\C}))=\MT_2(X_\C)$, it follows that $f^{-1}(\MT_2(X_\C)) = \MT(A_{\C})$.
    By our assumption, the Galois representation $\phi^{\KS}_\ell$ takes values in $f^{-1}(\MT_2(X_\C))$, which proves (1).

(2) For each $\ell$, we define $ \epsilon_\ell \colon G_E \to \langle s \rangle $ by $\epsilon_\ell(\sigma)=\phi^{\MT}_\ell(\sigma) \cdot (R\circ \phi^{\KS}_\ell)(\sigma)^{-1}$.
We have to show that $\epsilon_\ell$ is independent of $\ell$.
For this, it suffices to show that the subgroup $\ker(\epsilon_\ell) \subset G_E$ is independent of $\ell$.

Let $J_\ell \subset \MT(A_\C)_{\Q_\ell} \times\bG_{\Q_\ell}$ be the Zariski closure of the image of
\[
  \tau_\ell\coloneqq 
  (\phi_\ell^{\KS}, \phi_\ell^{\MT})
  \colon G_E \longrightarrow
  \MT(A_\C)(\Q_\ell)\times\bG(\Q_\ell).
\]
We claim that $\tau^{-1}_\ell(J^\circ_\ell(\Q_\ell)) = \ker(\epsilon_\ell)$.
This implies that $\ker(\epsilon_\ell)$ is independent of $\ell$ by a theorem of Serre \cite{RibetLetter} (see also \cite[Proposition~(6.14)]{LarsenPink92}).
To show the claim, let $\Gamma_R\subset\MT(A_\C)\times\bG$ be the graph of $R$.
Since $\bG \to \MT_2(X_\C)$ is an isogeny, the degree-two comparison shows that $ J_\ell^\circ\subset(\Gamma_R)_{\Q_\ell}; $ see also the proof of \cite[Proposition~4.7]{ArithmeticNagai}.
Since $R\colon\MT(A_\C)\to\bG$ is an isogeny, \Cref{thm: MT conjecture} implies that $J_\ell^\circ$ and $(\Gamma_R)_{\Q_\ell}$ have the same dimension.
Hence $J_\ell^\circ=(\Gamma_R)_{\Q_\ell}$.
Consequently, we have $\tau^{-1}_\ell(J^\circ_\ell(\Q_\ell)) = \ker(\epsilon_\ell)$.
This completes the proof of (2).
\end{proof}

\begin{thm}\label{Theorem: strong compatibility of MT representation}
Let $X$ be a \hk variety over a number field $F$ with $b_2(X)\geq4$.
Assume that the algebraic monodromy group of $H^\bullet_\et(X_{\overline F},\Q_\ell)$ is connected for every $\ell$, and that $X$ admits a Kuga--Satake abelian variety $A$ over $F$ with semistable reduction at a finite place $v\mid p$ of $F$.
Let $\bG\coloneqq \MT(X_\C)$ and consider the associated $\bG$-valued Weil--Deligne representation $\rho^{\MT}_{\ell,v}$ of $F_v$ over $\overline{\Q}_\ell$.
Then $\{\rho^{\MT}_{\ell,v}\}_\ell$ is rationally strongly compatible in the sense of \Cref{Definition:rational strong compatibility}.
\end{thm}

\begin{proof}
By \Cref{Lemma: LLV and Kuga-Satake hodge structure} and \Cref{Proposition:LLV and Kuga-Satake}, there are an isogeny $R\colon \MT(A_\C) \to\bG$ over $\Q$ and a continuous character $\epsilon\colon G_F\to\langle s\rangle$ such that $ R\circ\phi^{\KS}_\ell=\epsilon\phi^{\MT}_\ell $ for every $\ell$.
By \Cref{Theorem:strongly compatiblity for abelian variety}, there exists an $\MT(A_\C)$-valued Weil--Deligne representation $\rho^A_v=(r^A_v,N^A_v)$ of $F_v$ over $\overline{\Q}$ realizing the resulting rationally strongly compatible system.
Write $\epsilon_v$ for the restriction of $\epsilon$ to $W_{F_v}$ and set
\[
\rho_v\coloneqq (\epsilon_v\cdot(R\circ r^A_v),\,dR(N^A_v) ).
\]
Since $R$ is defined over $\Q$ and $\epsilon_v$ takes values in $Z(\bG)(\Q)$, the $\bG(\overline{\Q})$-conjugacy class of $\rho_v$ is $G_\Q$-invariant.
The constructions given in \Cref{Example: WD rep ell adic case,Example: WD rep p adic case} are compatible with tensor products and send the finite-order character $\epsilon|_{G_{F_v}}$ to $(\epsilon_v,0)$.
Together with functoriality under $R$, we see that
\[
(\rho_v)_{\overline{\Q}_\ell}
\sim_{\bG(\overline{\Q}_\ell)}
\rho^{\MT}_{\ell,v}
\]
for every $\ell$, including $\ell=p$.
This proves the assertion.
\end{proof}

\begin{rem}\label{Remark: local strong compatibility}
The assertion of \Cref{Theorem: strong compatibility of MT representation} also holds over a finite extension $K/\Q_p$ equipped with an embedding $K\hookrightarrow\C$.
More precisely, let $X$ be a \hk variety over $K$ with $b_2(X)\geq4$.
Assume that, for every $\ell$, the Galois representation on $H^\bullet_\et(X_{\overline K},\Q_\ell)$ takes values in $\MT(X_\C)(\Q_\ell)$, and that $X$ admits a Kuga--Satake abelian variety over $K$ with semistable reduction.
Then the associated $\MT(X_\C)$-valued Weil--Deligne representations form a rationally strongly compatible system.
(Here the notions of a Kuga--Satake abelian variety and rational strong compatibility are understood by extending the definitions above to this local setting in the obvious way.)
Indeed, the comparison with Kuga--Satake abelian varieties also holds in this setting by descent to a finitely generated subfield of $K$ and then the same proof applies using the local version of \Cref{Theorem:strongly compatiblity for abelian variety}; see \cite[Remark~5.3.13]{KZb}.
\end{rem}

\begin{cor}\label{Corollary: independence of Frobenius eigenvalues}
Let $K$ be a finite extension of $\Q_p$, and let $X$ be a \hk variety over $K$.
For every $\ell$ and $i$, let $\rho_{\ell,i}=(r_{\ell,i},N_{\ell,i})$ be the Weil--Deligne representation of $K$ over $\overline{\Q}_\ell$ associated with $H^i_\et(X_{\overline K},\Q_\ell)$.
There exists a finite extension $L/K$ such that, for every $i$ and every geometric Frobenius lift $\Frob_L\in W_L$, the characteristic polynomial $ \det (T-r_{\ell,i}(\Frob_L)) $ belongs to $\Z[T]$ and is independent of $\ell$, including $\ell=p$.
If $b_2(X)=3$, one may take $L=K$.
\end{cor}

\begin{proof}
Fix an embedding $\overline{K}\hookrightarrow\C$.
Suppose first that $b_2(X)=3$.
Then $\MT(X_\C)$ is a torus.
Thus, by \Cref{thm: MT conjecture}, we see that $H^i_\et(X_{\overline K},\Q_\ell)$ is potentially unramified if $\ell \neq p$ and potentially crystalline if $\ell = p$ for all $i$.
By the weight--monodromy theorem (\Cref{Theorem:purity of Weil-Deligne representations}), all the Frobenius eigenvalues in degree $i$ are therefore of weight $i$.
By \cite[Theorems~B and~D]{Ochiai}, the alternating traces of all positive powers of $\Frob_K$ are rational and independent of $\ell$, including $\ell=p$.
Thus the alternating product of characteristic polynomials is rational and independent of $\ell$.
Since different cohomological degrees have distinct weights, each factor has the same properties as well.

Suppose now that $b_2(X)\geq4$.
After a finite extension $L/K$, the Galois representations are valued in the Mumford--Tate group, and $X_L$ admits a Kuga--Satake abelian variety over $L$ with semistable reduction.
The assertion then follows from \Cref{Remark: local strong compatibility}.

Finally, the Frobenius eigenvalues for $\ell\neq p$ are algebraic integers by \cite[Proposition~2.1]{Ochiai}.
Hence the common characteristic polynomials belong to $\Z[T]$ for all $\ell$ (including $\ell =p$).
\end{proof}

The following example shows that the character $\epsilon$ in \Cref{Proposition:LLV and Kuga-Satake} need not be trivial in general.

\begin{ex}\label{ex:nontrivial central inertia character}
Let $F$ be a number field, let $v\mid 3$, and let $C$ be an elliptic curve over $F$ such that $ \End_F(C)_\Q=\End_{\overline F}(C_{\overline F})_\Q\simeq\Q(\sqrt{-1})$ and all $3$-torsion points of $C$ are rational over $F$.
Suppose that there is a quadratic character $\epsilon \colon G_F\to\{\pm1\}$, ramified at $v$, such that the quadratic twist $C^\epsilon$ of $C$ by $\epsilon$ has good reduction at $v$.
Then for every $\ell \neq 3$, the inertia subgroup $I_v$ acts on $H^1_{\et}(C_{\overline F}, \Q_\ell)$ through the character $\epsilon$.

Let $S\coloneqq  C \times C$ and let $X=K_2(S)$ be the associated generalized Kummer fourfold over $F$.
We have $G_F$-equivariant decompositions
\[
\begin{aligned}
H^2_\et(X_{\overline F}, \Q_\ell)&\simeq H^2_\et(S_{\overline F}, \Q_\ell)\oplus\Q_\ell(-1),\\
H^3_\et(X_{\overline F}, \Q_\ell)&\simeq H^3_\et(S_{\overline F}, \Q_\ell)\oplus H^1_\et(S_{\overline F}, \Q_\ell)(-1),\\
H^4_\et(X_{\overline F}, \Q_\ell)&\simeq\Sym^2 H^2_\et(X_{\overline F}, \Q_\ell)\oplus \Q_\ell(-2)^{ \oplus 80}
\end{aligned}
\]
and similarly for Betti cohomology groups.
For the third identification, see \cite[Theorem~1.1]{FH23}.
These decompositions are obtained from algebraic correspondences defined over $F$ and are compatible with the \'etale--Betti comparison.

Put
\[
T\coloneqq \MT(C_\C)\simeq
\operatorname{Res}_{\Q(\sqrt{-1})/\Q}\GG_m.
\]
Since all endomorphisms of $C_{\overline{F}}$ are defined over $F$, the Galois representation on $H^1_\et(C_{\overline F},\Q_\ell)$ takes values in $T(\Q_\ell)$ for every $\ell$.
The above decompositions give a representation
\[
\Phi\colon T\longrightarrow
\prod_j\GL(H^j_B(X_\C,\Q))
\]
such that $\Phi(T)=\MT(X_\C)$ and, for every $\ell$, the composition
\[
G_F \longrightarrow T(\Q_\ell)
\xlongrightarrow{\Phi}
\prod_i\GL(H^i_\et(X_{\overline F},\Q_\ell))
\]
is the usual Galois representation on the total cohomology of $X$.
In particular, the algebraic monodromy group of $X$ is connected (for every $\ell$).

Finally, let $U\subset V^2_B(1)$ be the orthogonal complement of $\Pic(X)_\Q=\Pic(X_{\overline{F}})_\Q$.
The CM structure gives
\[
T\simeq \GSpin(U)\longrightarrow \GSpin(V^2_B(1)).
\]
Since $C^\epsilon$ and $C$ are isomorphic over $\C$, they have the same $T$-representation on $H^1_B$.
Thus there exists a $T$-equivariant isomorphism
\[
\theta \colon H^1_B((C^\epsilon)_{\C},\Q)^{\oplus 64} \isomto \Cl(V^2_B(1)).
\]
We set $A\coloneqq (C^\epsilon)^{64}$.
By construction, one can check that $A$, equipped with the isomorphism $\theta$, is a Kuga--Satake abelian variety over $F$ in the sense of \Cref{Definition: Kuga-Satake abelian variety}.
Since $A$ has good reduction at $v$,
the inertia subgroup $I_v$ acts trivially on $H^1_\et(A_{\overline F},\Q_\ell)$
for $\ell\neq3$.
On the other hand, it acts on $H^3_\et(X_{\overline F},\Q_\ell)$ through the nontrivial character $\epsilon|_{I_v}$.
Thus the character in \Cref{Proposition:LLV and Kuga-Satake} is nontrivial.
\end{ex}

\subsection{An integral refinement of strong compatibility}\label{subsection:An integral refinement of strong compatibility}

Keep the notation and hypotheses of \Cref{Subsection:definitions strong compatibility}.
In particular, $X$ denotes an abelian variety over $F$, or a \hk variety over $F$.
Assume in addition that $b_2(X)\geq4$
if $X$ is a \hk variety.
Put
\[
\Lambda_B=
\begin{cases}
H^1_B(X_\C,\Z), & \text{if $X$ is an abelian variety},\\
H^\bullet_B(X_\C,\Z), & \text{if $X$ is a \hk variety}.
\end{cases}
\]
Let $\bG\coloneqq \MT(X_\C)$.
Let $v \mid p$ be a finite place of $F$ and let $q$ be the order of the residue field of $F_v$.
If $X$ is an abelian variety, assume that it has semistable reduction at $v$.
If $X$ is a \hk variety, assume that it admits a Kuga--Satake abelian variety over $F$ with semistable reduction at $v$.

By \Cref{Theorem:strongly compatiblity for abelian variety,Theorem: strong compatibility of MT representation}, we may choose a Frobenius semisimple $\bG$-valued Weil--Deligne representation $\rho_v=(r_v,N_v)$ over $\overline{\Q}$ such that
\[
(\rho_v)_{\overline{\Q}_\ell}
\sim_{\bG(\overline{\Q}_\ell)}
\rho^{\MT}_{\ell,v}
\]
for every $\ell$.
We note that $r_v(I_{F_v})$ is trivial in the abelian case
by the semistable reduction assumption.
In the hyper-K\"ahler case, it is contained in
$Z(\bG)(\overline{\Q})$ by
\Cref{Proposition:LLV and Kuga-Satake}.

\begin{defn}\label{Definition:the integer D_0}
    We fix a positive integer $D_0$ divisible by $p$, with the following properties:
\begin{enumerate}[label=(\alph*)]
\item $\Lambda_{B,\Z[1/D_0]}\coloneqq \Lambda_B\otimes_\Z\Z[1/D_0]$ is torsion-free.
\item The schematic closure $\cG\subset\GL(\Lambda_{B,\Z[1/D_0]})$ of $\bG$ is a reductive group scheme over $\Z[1/D_0]$.
\item Every prime $\ell\nmid D_0$ satisfies $\ell>b_1(X)$ when $X$ is an abelian variety and $\ell > \max_i  b_i(X)$ when $X$ is a \hk variety.
\item $q-1\in\Z[1/D_0]^\times$ and for every geometric Frobenius lift $\Frob_q\in W_{F_v}$, any two distinct eigenvalues $a,b$ of $r_v(\Frob_q)$ on $H^i_B(X_\C,\Q)\otimes_\Q\overline{\Q}$ satisfy
$
a-b\in\overline{\Z}[1/D_0]^\times,
$
where $i=1$ if $X$ is an abelian variety and $i$ is arbitrary if $X$ is a \hk variety.
\item $\rho_v$ extends to a $\cG$-valued Weil--Deligne representation of $F_v$ over $\overline{\Z}[1/D_0]$:
\[
\rho^{\integ}_v=(r^{\integ}_v,N^{\integ}_v).
\]
Here $\overline{\Z}$ is the ring of algebraic integers in $\overline{\Q}$.
\end{enumerate}
These conditions can be achieved by taking $D_0$ sufficiently divisible.
\end{defn}

For a prime $\ell\nmid D_0$, the representation $\phi^{\MT}_\ell$ takes values in $\cG(\Z_\ell)$.
By \Cref{Example: integral Weil-Deligne construction} and \Cref{Proposition: independence of integral WD construction}, we obtain the associated $\cG$-valued Weil--Deligne representation
\[
\rho^{\MT,\integ}_{\ell,v}
=(r^{\MT,\integ}_{\ell,v},N^{\MT,\integ}_{\ell,v})
\]
of $F_v$ over $\overline{\Z}_\ell$, whose equivalence class is independent of all choices in the construction.
Its generic fiber is equivalent to $\rho^{\MT}_{\ell,v}$.

\begin{prop}
\label{Proposition: ell independence of Weil part of integral MT representation}
Let $\ell\nmid D_0$ be a prime.
The representations $(r^{\integ}_v)_{\overline{\Z}_\ell}$ and $r^{\MT,\integ}_{\ell,v}$ are $\cG(\overline{\Z}_\ell)$-conjugate.
Here $(r^{\integ}_v)_{\overline{\Z}_\ell}$ denotes the base change of $r^{\integ}_v$ along the map $\overline{\Z}[1/D_0] \to \overline{\Z}_\ell$ induced by $\iota_\ell$.
Moreover, their scheme-theoretic centralizers in $\cG_{\overline{\Z}_\ell}$ are smooth over $\overline{\Z}_\ell$.
\end{prop}

\begin{proof}
Fix a geometric Frobenius lift $\Frob_q\in W_{F_v}$, and put
\[
x\coloneqq (r^{\integ}_v)_{\overline{\Z}_\ell}(\Frob_q),
\quad
y\coloneqq r^{\MT,\integ}_{\ell,v}(\Frob_q).
\]
Their generic fibers are semisimple and $\bG(\overline{\Q}_\ell)$-conjugate by rational strong compatibility.
The centralizers of $x$ and $y$ in $\cG_{\overline{\Z}_\ell}$ are smooth and contain maximal tori $T_x$ and $T_y$ of $\cG_{\overline{\Z}_\ell}$ by \Cref{Lemma: smoothness of Frobenius centralizer}.
Since a maximal torus is its own centralizer, we have $x\in T_x(\overline{\Z}_\ell)$ and $y\in T_y(\overline{\Z}_\ell)$.
By \cite[Theorem~3.2.6]{Conrad2014}, these tori are $\cG(\overline{\Z}_\ell)$-conjugate, so we may assume that $x,y\in T(\overline{\Z}_\ell)$ for a common maximal torus $T$.
Their generic fibers then lie in the same orbit with respect to the action of the Weyl group associated with $T$.
The corresponding element in the Weyl group lifts to an element $g \in N_{\cG_{\overline{\Z}_\ell}}(T)(\overline{\Z}_\ell)$ since the Weyl group (scheme) is finite \'etale and $\overline{\Z}_\ell$ is strictly henselian.
It follows that $gxg^{-1}=y$.

The images of the inertia subgroup for both representations are central, and rational strong compatibility implies that their restrictions to the inertia subgroup are the same.
It follows that $(r^{\integ}_v)_{\overline{\Z}_\ell}$ and $r^{\MT,\integ}_{\ell,v}$ are conjugate by the element $g$.
Their centralizers equal those of $x$ and $y$, respectively, and hence are smooth over $\overline{\Z}_\ell$.
\end{proof}

We shall prove that there is a positive integer $D$, divisible by $D_0$, such that for every prime $\ell\nmid D$, both $(\rho_v^{\integ})_{\overline{\Z}_\ell}$ and $\rho^{\MT,\integ}_{\ell,v}$ are pure in the sense of \Cref{Definition:integral G-valued purity}.
The key ingredient for this is the following torsion-freeness result.

\begin{prop}\label{Proposition: torsion freeness of cohomological monodromy}
Let $K$ be a finite extension of $\Q_p$, and let $X$ be an abelian variety or a \hk variety over $K$.
Fix integers $i,m,n \geq 0$, and consider the $G_K$-module
\[
H^i_{\et}(X_{\overline K},\Z_\ell)^{\otimes m}\otimes_{\Z_\ell}
H^i_{\et}(X_{\overline K},\Z_\ell)^{\vee, \otimes n}.
\]
For all but finitely many $\ell \neq p$, the monodromy operator $N_\ell$ of the Weil--Deligne representation $\rho_\ell=(r_\ell, N_\ell)$ of $K$ over $\overline{\Z}_\ell$ associated with this $G_K$-module (\Cref{Example: integral Weil-Deligne construction}) has the following property \textup{(t-f)}: For all $j \geq 0$, the cokernel of $(N_\ell)^{j}$ is torsion-free.
\end{prop}

\begin{proof}
The property \textup{(t-f)} is preserved by duals and tensor products after excluding finitely many further primes.
Indeed, by \cite[Lemma~2.10]{Ito21} and Nakayama's lemma,
\textup{(t-f)} is equivalent to compatibility of the monodromy
filtration with reduction to $\overline{\F}_\ell$.
This compatibility is preserved by duals and tensor products,
since the formulas given in \cite[Proposition~(1.6.9)]{WeilII}
remain valid over $\overline{\F}_\ell$
outside a finite set of primes depending only on
the ranks of the underlying modules.
Thus it suffices to consider the Weil--Deligne representation $\rho_\ell=(r_\ell, N_\ell)$ of $K$ over $\overline{\Z}_\ell$ associated with $H^i_{\et}(X_{\overline K},\Z_\ell)$.
If $X$ is an abelian variety, then the assertion is proved in \cite{Ito21}; more precisely, this follows from \cite[Theorem~3.7 and Proposition~3.10(ii)]{Ito21}.
We assume that $X$ is a \hk variety over $K$.
If $i=2$, then this again follows from \cite{Ito21}.
We shall deduce the general case from this case.

We choose an embedding $K \hookrightarrow \C$ and let $\bG\coloneqq \MT(X_\C)$.
We may assume that $K=K^c$ and obtain the representation $\phi^{\MT}_\ell \colon G_K \to \bG(\Q_\ell)$ for all $\ell$ as in \Cref{Subsection:MT-WD representation}.
By \Cref{lem:weighted tensor generation}, $H^i_B(X_\C,\Q)^{\otimes2}$ is a direct summand of a finite direct sum of $H^2_B(X_\C,\Q)^{\otimes a} \otimes H^2_B(X_\C,\Q)^{\vee, \otimes b}$ for some $a, b$ as a $\bG$-representation.
From this, we see that $H^i_{\et}(X_{\overline K},\Z_\ell)^{\otimes2}$ is a direct summand of a finite direct sum of $H^2_{\et}(X_{\overline K},\Z_\ell)^{\otimes a} \otimes H^2_{\et}(X_{\overline K},\Z_\ell)^{\vee, \otimes b}$ as a $G_K$-module for all but finitely many $\ell \neq p$.
It follows that $\rho_\ell^{\otimes 2}$ satisfies property \textup{(t-f)}.
Finally, property \textup{(t-f)} can be detected on tensor squares
for sufficiently large $\ell$.
Indeed, the monodromy filtration on a tensor square is the
tensor product filtration, both in characteristic zero and,
for all but finitely many $\ell$, after reduction to
$\overline{\F}_\ell$.
Since a filtration is determined by its tensor square,
compatibility of the monodromy filtration with reduction
for $\rho_\ell^{\otimes2}$ implies the same compatibility
for $\rho_\ell$.
Thus $\rho_\ell$ also has property \textup{(t-f)}.
This completes the proof.
\end{proof}

\begin{rem}\label{Remark:torsion WMC}
The torsion analogue of the weight--monodromy conjecture formulated in \cite[Conjecture~3.5]{Ito21} also holds for hyper-K\"ahler varieties over $K$ in every degree.
Indeed, this follows from \Cref{Theorem:purity of Weil-Deligne representations} and \Cref{Proposition: torsion freeness of cohomological monodromy} together with \cite[Proposition~3.10(iii)]{Ito21}.
\end{rem}

We return to the notation and assumptions of this subsection.
Recall that $\mu \colon \GG_m \to \bG$ is the weight cocharacter.
This extends to a cocharacter of $\cG$, which we denote by the same notation $\mu$.

\begin{prop}\label{Proposition: integral purity of MT representations}
There exists a positive integer $D$ divisible by $D_0$, such that, for every prime $\ell\nmid D$, both
\[
(\rho_v^{\integ})_{\overline{\Z}_\ell}
\quad\text{and}\quad
\rho^{\MT,\integ}_{\ell,v}
\]
are pure with respect to $\mu$ in the sense of \Cref{Definition:integral G-valued purity}.
\end{prop}

\begin{proof}
We want to apply \Cref{Lemma: sufficient criterion for integral purity}.
By the preceding results (and the corresponding known results for abelian varieties), the generic fibers are Frobenius semisimple and pure with respect to $\mu$.
Condition (d) implies that the Frobenius-weight cocharacters are defined over $\overline{\Z}_\ell$ for every $\ell \nmid D_0$.
Moreover, the scheme-theoretic centralizers of 
the corresponding actions of $W_{F_v}$ are smooth over $\overline{\Z}_\ell$ by \Cref{Proposition: ell independence of Weil part of integral MT representation}.
Thus conditions \textup{(1)--(3)} of \Cref{Definition:integral G-valued purity} hold.

Let $\ge\coloneqq  \Lie(\cG)$.
Since $\rho_v$ is defined over $\overline{\Q}$, it is easy to see that there exists a positive integer $D$ divisible by $D_0$, such that, for every prime $\ell\nmid D$, the isotypic decomposition of $\ge_{\overline{\Q}_\ell}$ for the action of $r^{\integ}_v(W_{F_v})$ is defined over $\overline{\Z}_\ell$.
By \Cref{Proposition: ell independence of Weil part of integral MT representation}, the same holds for $r^{\MT,\integ}_{\ell,v}$.

It remains to show that for all but finitely many $\ell$, the cokernels of
\[
d_1 \colon \ge_{\overline{\Z}_\ell} \longrightarrow \ge_{\overline{\Z}_\ell}, \quad Y \mapsto [Y, N^{\integ}_v]
\]
and
\[
d_2 \colon \ge_{\overline{\Z}_\ell} \longrightarrow \ge_{\overline{\Z}_\ell}, \quad Y \mapsto [Y, N^{\MT, \integ}_{\ell, v}]
\]
are torsion-free.
Since $\rho_v$ is defined over $\overline{\Q}$, one can easily check the assertion for $d_1$.
For $d_2$, we note that
the natural
$\bG$-equivariant inclusion
\[
\ge_\Q\hookrightarrow
\bigoplus_i\End(H^i_B(X_\C,\Q))
\]
splits since $\bG$ is reductive.
After enlarging $D$, we may assume that the inclusion and splitting extend
over $\Z[1/D]$.
Under the \'etale--Betti comparison, this splitting is also $G_{F_v}$-equivariant.
Now \Cref{Proposition: torsion freeness of cohomological monodromy} implies that after enlarging $D$ further, the cokernel of $d_2$ is torsion-free for every $\ell \nmid D$.

All the hypotheses of \Cref{Lemma: sufficient criterion for integral purity} are now satisfied for both representations.
\end{proof}

\begin{thm}
\label{Theorem:integral compatibility for HK}
There exists a positive integer $D$ divisible by $D_0$, such that, for every prime $\ell\nmid D$, we have
\[
(\rho_v^{\integ})_{\overline{\Z}_\ell}
\sim_{\cG(\overline{\Z}_\ell)}
\rho^{\MT,\integ}_{\ell,v}.
\]
\end{thm}

\begin{proof}
By \Cref{Proposition: integral purity of MT representations}, we may choose $D$ divisible by $D_0$ such that both representations are pure with respect to $\mu$ for every $\ell \nmid D$.
Their generic fibers are Frobenius semisimple, and the representations $(r^{\integ}_v)_{\overline{\Z}_\ell}$ and $r^{\MT,\integ}_{\ell,v}$ are $\cG(\overline{\Z}_\ell)$-conjugate by \Cref{Proposition: ell independence of Weil part of integral MT representation}.
The assertion therefore follows from \Cref{Proposition:uniqueness of monodromy operator integral case}.
\end{proof}

\subsection*{Acknowledgements}
We thank Ziquan Yang for drawing our attention to the papers \cite{KZa,KZb}.
We also thank Salvatore Floccari, Tetsushi Ito, Teruhisa Koshikawa, and Teppei Takamatsu for helpful discussions, particularly on the LLV representation, strong compatibility, and the $\ell$-independence of Frobenius characteristic polynomials.
We are grateful to Zhiyuan Li, Ben Moonen, and Ziquan Yang for their valuable comments on an early draft of this manuscript.

\subsection*{Funding information}
K. Ito is supported by JSPS KAKENHI Grant Numbers 24K16887 and 24H00015.
H. Zou is supported by the Deutsche Forschungsgemeinschaft (DFG, German Research Foundation), Project-ID 491392403, TRR 358.

\subsection*{AI declarations}
This project began in early 2026, and most of the results presented here, including the weight--monodromy conjecture and the strong compatibility, were obtained before August 2026 without any assistance from generative AI tools.
Thereafter, Codex and Claude Code were used to assist with exposition, literature searches, statement simplification, and proof audits.
The uses described below bear on the mathematical content.

The second author arrived at the statement of \Cref{lem:global additive rigidity} after completing his project with Zhichao Tang \cite{TZ26}, and used it to rule out the case in which the unipotent radical is isomorphic to the additive group $\GG_a$; at the time he did not realize that a full proof of semisimplicity was within reach.
The reduction of a general unipotent quotient to the case of $\GG_a$, given in \Cref{prop:no-unipotent quotient for global representation}, was pointed out by ChatGPT 5.6 Sol in August 2026, while he was asking about an unrelated general fact on algebraic groups during the Algebraic Geometry Summer School at SCMS, Shanghai.
With the Mumford--Tate conjecture then available in full,
rather than only in its semisimple version,
several technical workarounds could be dropped. In preparing \Cref{Strongly compatible systems with integral coefficients}, the first author used ChatGPT 6 Astra to help clarify how to control the finite extensions of base fields required in the argument.
In particular, the proof of \Cref{Proposition:LLV and Kuga-Satake} and the example in \Cref{ex:nontrivial central inertia character} were found with its help.

All AI outputs bearing on the mathematical content were verified independently by the authors.

\begingroup \sloppy
\printbibliography
\endgroup
\end{document}